\documentclass{tac}

\usepackage[english, french]{babel}
\usepackage[utf8]{inputenc}
\usepackage[T1]{fontenc}

\usepackage{
graphics,
latexsym,
textcomp,
amsfonts,
amssymb,
amsbsy,
amsmath,
}

\usepackage[hidelinks]{hyperref}

\usepackage[all,2cell,cmtip]{xy} \UseAllTwocells 
\SilentMatrices
\xyoption{v2}

\DeclareUnicodeCharacter{00A0}{~}

\newcommand\EnsSimp{\widehat{\Delta}}

\newcommand\DeuxFoncteurStrict[0]{\deux{}foncteur strict}
\newcommand\DeuxFoncteursStricts[0]{\deux{}foncteurs stricts}
\newcommand\DeuxTransformationLax[0]{transformation}
\newcommand\DeuxTransformationsLax[0]{transformations}
\newcommand\DeuxTransformationOpLax[0]{optransformation}

\newcommand\DeuxTransformationStricte[0]{transformation stricte}
\newcommand\DeuxTransformationsStrictes[0]{transformations strictes}
\newcommand\DeuxFoncteurLax[0]{\deux{}foncteur lax}
\newcommand\DeuxFoncteursLax[0]{\deux{}foncteurs lax}
\newcommand\DeuxFoncteurCoLax[0]{\deux{}foncteur colax}
\newcommand\DeuxFoncteursCoLax[0]{\deux{}foncteurs colax}

\newcommand\TrancheLax[3]{{#1}/_{\negmedspace {\rm{l}}}^{#2}{#3}}

\newcommand\TrancheCoLax[3]{{#1}/_{\negmedspace {\rm{c}}}^{#2}{#3}}

\newcommand\OpTrancheCoLax[3]{{#3}{\backslash}_{\mspace{-.3mu}{\rm{c}}}^{\mspace{-6.mu}#2}{#1}}

\newcommand\OpTrancheLax[3]{{#3}{\backslash}_{\mspace{-.3mu}{\rm{l}}}^{\mspace{-6.mu}#2}{#1}}

\newcommand\DeuxFoncTrancheLax[2]{{#1}/_{\negmedspace {\rm{l}}}{#2}}

\newcommand\DeuxFoncTrancheCoLax[2]{{#1}/_{\negmedspace {\rm{c}}}{#2}}

\newcommand\DeuxFoncOpTrancheLax[2]{{#2}\backslash_{\mspace{-.3mu}{\rm{l}}}{#1}}

\newcommand\DeuxFoncOpTrancheCoLax[2]{{#2}\backslash_{\mspace{-.3mu}{\rm{c}}}{#1}}

\newcommand\DeuxFoncTrancheLaxCoq[3]{{#1}/_{\negmedspace {\rm{l}}}^{#2}{#3}}

\newcommand\Fibre[3]{{#2}^{-1}({#3})}

\newcommand\CompDeuxUn[0]{}

\newcommand\CompDeuxZero[0]{\circ}

\newcommand\DeuxCellStructComp[3]{{#1}_{{#2},{#3}}}

\newcommand\DeuxCellStructId[2]{{#1}_{{#2}}}

\newcommand\Objets[1]{Ob({#1})}

\newcommand\UnCell[1]{Fl({#1})}

\newcommand\TildeLax[1]{\widetilde{#1}}
\newcommand\BarreLax[1]{\overline{#1}}

\newcommand\LaxCanonique[1]{\eta_{{#1}}}
\newcommand\StrictCanonique[1]{\epsilon_{{#1}}}
\newcommand\TransLaxCanonique[0]{\eta}
\newcommand\TransStrictCanonique[0]{\epsilon}

\newcommand\DeuxInt[1]{\int{#1}}
\newcommand\DeuxIntOp[1]{\int_{#1}^{op}}

\newcommand\deux{$2$\nobreakdash-}
\newcommand\un{$1$\nobreakdash-}

\newcommand\CatHom[3]{\underline{Hom}_{#1}(#2, #3)}
\newcommand\EnsHom[3]{{Hom}_{#1}(#2, #3)}

\newcommand\Hot{{\mathcal{H}\mspace{-2.mu}\it{ot}}}
\newcommand\Cat{{\mathcal{C}\mspace{-2.mu}\it{at}}}

\newcommand\DeuxCat{\text{$2$-$\Cat$}}

\newcommand\DeuxCatLax{2\hbox{\protect\nobreakdash-}\kern1pt\Cat_{lax}}
\newcommand\DeuxCatLaxNor{2\hbox{\protect\nobreakdash-}\kern1pt\Cat_{lax,nor}}

\newcommand\DeuxCatDeuxCat{\underline{2Cat}}

\newcommand\UnLocFond[1]{{#1}}
\newcommand\DeuxLocFond[1]{\mathcal{#1}}

\newcommand\DeuxLocFondLaxInduit[1]{\mathcal{#1}_{lax}}

\newcommand\Localisation[2]{{#2}^{-1}{#1}}

\newcommand\NerfLaxNor{N_{l,n}}
\newcommand\NerfLax{N_{l}}

\newcommand\NerfHom{\underline{N}}

\newcommand\SupHom{\underline{sup}}
\newcommand\SupHomObjet[1]{\underline{sup}_{#1}}

\newcommand\DeuxCatUnOp[1]{{#1}^{op}}
\newcommand\DeuxCatDeuxOp[1]{{#1}^{co}}
\newcommand\DeuxCatToutOp[1]{{#1}^{coop}}

\newcommand\DeuxFoncUnOp[1]{{#1}^{op}}
\newcommand\DeuxFoncDeuxOp[1]{{#1}^{co}}
\newcommand\DeuxFoncToutOp[1]{{#1}^{coop}}

\newcommand\DeuxFoncteurTranche[1]{T^{#1}}

\newcommand\mathdeuxcat[1]{\mathcal{#1}}
\newcommand\DeuxCatPonct{e}

\makeatletter 

\def\newtheoremp#1{\@ifnextchar[{\@newthmp{#1}}{\@newthmp{#1}[]}}
\def\@newthmp#1[#2]#3{\@ifnextchar[{\@@newthmp{#1}[#2]{#3}}{\@@newthmp{#1}[#2]{#3}[ ]}}
\def\@@newthmp#1[#2]#3[#4]{\expandafter\def\csname #1\endcsname{%
\staterm{#3}\@ifnextchar[{\@credit}{}}%
 \expandafter\let\csname
     end#1\endcsname\endstate}

\def\thebibliography#1{\section*{Références\@mkboth
 {REFERENCES}{REFERENCES}}\list
 {[\arabic{enumi}]}{\settowidth\labelwidth{[#1]}\leftmargin\labelwidth
 \advance\leftmargin\labelsep
 \usecounter{enumi}}
 \def\newblock{\hskip .11em plus .33em minus .07em}
 \sloppy\clubpenalty4000\widowpenalty4000
 \sfcode`\.=1000\relax}

\def\Pr@@f{\subsubsection*{\textsc{Démonstration.}}}

\newtheorem{theo}{Théorème}[section]
\newtheorem{prop}[theo]{Proposition}
\newtheorem{lemme}[theo]{Lemme}
\newtheorem{corollaire}[theo]{Corollaire}

\newtheoremrm{exemple}[theo]{Exemple}
\newtheoremrm{rem}[theo]{Remarque}

\newtheoremrm{df}[theo]{Définition}
\newtheoremrm{demo}{Démonstration}

\newtheoremp{paragr}[theo]{} 

\makeatother

\title{Un Théorème A de Quillen pour les $2$-foncteurs lax}

\author{Jonathan Chiche}

\address{20, rue des Angles, 94130 Nogent-sur-Marne}
\eaddress{jonathan.chiche@polytechnique.org}
\keywords{homotopy theory, 2-categories, Quillen's Theorem A}
\amsclass{18D05, 18G55, 55U35}

\copyrightyear{2014}

\begin{document}

\maketitle

\begin{abstract}
On généralise le Théorème A de Quillen aux triangles de \DeuxFoncteursLax{} commutatifs à transformation près. Un cas particulier de ce résultat permet d'établir que les \deux{}catégories modélisent les types d'homotopie. 
\end{abstract}

\medskip

\selectlanguage{english}

\begin{abstract}
We generalize Quillen's Theorem A to triangles of lax \deux{}functors which
commute up to transformation. It follows from a special case of this result
that \deux{}categ\-ories are models for homotopy types.
\end{abstract}

\selectlanguage{french}

\section{Introduction}

Dans \emph{Pursuing Stacks} \cite{Poursuite} et \emph{Les dérivateurs} \cite{Derivateurs}, Grothendieck considère $\Cat$, la catégorie des petites catégories, comme le « ``paradis originel'' pour l'algèbre topologique » (voir aussi \cite{LettreGrothendieckThomason}). Il montre que toutes les constructions homotopiques usuelles peuvent s’effectuer dans $\Cat$ de façon très naturelle, et souvent plus simplement que dans la catégorie des espaces topologiques $Top$ ou celle des ensembles simpliciaux $\EnsSimp$. Dans cette théorie de l'homotopie de $\Cat$, un rôle crucial est assuré par le Théorème A de Quillen, provenant du travail fondateur de la K-théorie algébrique supérieure de ce dernier \cite{QuillenK}.

D'importants efforts ont été consacrés ces dernières années à l'étude de la théorie de l’homotopie des catégories supérieures. Dans la catégorie $\DeuxCat$, dont les objets sont les \deux{}catégories strictes et les morphismes les \DeuxFoncteursStricts{}, se distinguent trois définitions intéressantes d’équivalence faible. La plus forte est celle d'équivalence de \deux{}catégories. Une notion moins restrictive est celle d'équivalence « de Dwyer-Kan », \deux{}foncteur induisant des équivalences faibles entre les nerfs des catégories de morphismes et une équivalence des catégories obtenues des \deux{}catégories en remplaçant les catégories de morphismes par leur $\pi_{0}$. La catégorie homotopique obtenue en inversant ces équivalences faibles est, au moins conjecturalement, celle des $(\infty, 1)$\nobreakdash-catégories. On peut enfin s’intéresser, et c'est ce que nous ferons dans le présent article, aux équivalences faibles « de Thomason », \deux{}foncteurs induisant une équivalence faible des espaces classifiants. Aucune démonstration ne se trouve à notre connaissance publiée du fait que la catégorie localisée de $\DeuxCat$ relativement à ces équivalences faibles est équivalente à la catégorie homotopique classique. Dans l'introduction de \cite{WHPT}, Worytkiewicz, Hess, Parent et Tonks affirment construire une structure de catégorie de modèles sur $\DeuxCat$ Quillen\nobreakdash-équivalente à la structure de catégorie de modèles classique sur les ensembles simpliciaux. Comme l'ont remarqué Ara et Maltsiniotis \cite{AraMaltsiniotis}, plusieurs passages cruciaux de la démonstration de l'existence de cette structure de catégorie de modèles sur $\DeuxCat$ sont faux dans \cite{WHPT} (les mêmes remarques s'appliquent à l'article \cite{HPTW}). De plus, aucun argument, même incorrect, ne s'y trouve énoncé quant au fait que l'adjonction de Quillen prétendument construite est en fait une équivalence de Quillen. Dans \cite{AraMaltsiniotis}, Ara et Maltsiniotis démontrent notamment l'existence d'une adjonction de Quillen entre $\DeuxCat$ et la catégorie des ensembles simpliciaux munie de sa structure de catégorie de modèles classique. Les équivalences faibles de $\DeuxCat$ pour cette structure de catégorie de modèles sont précisément celles que nous considérons ici. Pour montrer que l'adjonction de Quillen construite est en fait une équivalence de Quillen, Ara et Maltsiniotis s'appuient sur le résultat principal que nous démontrons dans la dernière section du présent article. 

Des pas essentiels dans cette direction ont été franchis par Bullejos, Cegarra et del Hoyo. Le présent article repose sur leurs résultats. Dans \cite{BC}, Bullejos et Cegarra établissent une version \deux{}catégorique du Théorème A de Quillen (voir aussi \cite{Cegarra}), généralisant l'énoncé classique au cas des \DeuxFoncteursStricts{}. Dans \cite{TheseDelHoyo} (voir aussi \cite{ArticleDelHoyo}), del Hoyo, s'appuyant sur ce dernier résultat, étend l'énoncé classique au cas des \deux{}foncteurs lax normalisés de source une \un{}catégorie. Ces deux généralisations se restreignent au cas « absolu », c'est-à-dire sans base. Il en va de même d'une nouvelle généralisation, pour les \DeuxFoncteursLax{}, annoncée par del Hoyo dans \cite{NotesDelHoyo}, dont nous avons par la suite dégagé une variante relative par un argument différent. Suivant une suggestion de Maltsiniotis, nous présentons ici un formalisme permettant de dégager une version encore plus générale — et, peut-être, aussi générale qu'il est permis de l'espérer pour les \deux{}catégories —, mais en même temps très naturelle, du Théorème A de Quillen. 

Le résultat original de Quillen affirme qu'étant donné un foncteur $u : A \to B$, si la catégorie (« comma » ou « tranche ») $A/b$ est faiblement contractile pour tout objet $b$ de $B$, alors $u$ est une équivalence faible, c'est-à-dire que son nerf est une équivalence faible simpliciale. Le cas relatif, qui peut se démontrer de façon tout à fait analogue, s'énonce comme suit. 

\begin{theo}\label{ThACat}(Théorème A de Quillen pour les $1$-foncteurs)
Soit 
$$
\xymatrix{
A
\ar[rr]^{u}
\ar[dr]_{w}
&&B
\ar[dl]^{v}
\\
&C
}
$$
un diagramme commutatif dans $\Cat$. Supposons que, pour tout objet $c$ de $C$, le foncteur induit $A/c \to B/c$ soit une équivalence faible. Alors $u$ est une équivalence faible. 
\end{theo}

On peut d'ores et déjà remarquer que, les transformations naturelles — \deux{}cellules de la \deux{}catégorie $\Cat$ — pouvant s'interpréter comme des homotopies dans $\Cat$, l'on obtiendrait une version plus naturelle du théorème \ref{ThACat} en remplaçant le triangle commutatif de son énoncé par un triangle commutatif \emph{à transformation naturelle près} seulement. Ce résultat, qui semble absent de la littérature, mais démontré par Maltsiniotis dans \cite{CourrielGeorges}, est un cas particulier de celui que nous présentons. 

La formulation de la généralisation souhaitée de l'énoncé du théorème \ref{ThACat} ne présente guère de difficulté sérieuse. La notion de « \deux{}catégorie comma » pour un \DeuxFoncteurLax{} est dégagée depuis longtemps, de même que les généralisations plus ou moins strictes de celle de transformation naturelle. Nous appellerons « transformation stricte » la variante stricte, et « transformation » et « optransformation » les deux variantes faibles, duales l'une de l'autre. Le résultat s'énonce alors comme suit.

\begin{theo}
Soit
$$
\xymatrix{
\mathdeuxcat{A}  
\ar[rr]^{u} 
\ar[dr]_{w} 
&&\mathdeuxcat{B}
\dtwocell<\omit>{<7.3>\sigma} 
\ar[dl]^{v}
\\ 
& 
\mathdeuxcat{C}
&{}
}
$$
un diagramme dans lequel $u$, $v$ et $w$ sont des \DeuxFoncteursLax{} et $\sigma$ est une \DeuxTransformationOpLax{}. Supposons que, pour tout objet $c$ de $\mathdeuxcat{C}$, le \DeuxFoncteurLax{} induit
$
\mathdeuxcat{A} / {c} \to \mathdeuxcat{B} / {c}
$
soit une équivalence faible. Alors $u$ est une équivalence faible.
\end{theo}

La stratégie de démonstration que nous adoptons consiste à se ramener au cas d'un diagramme de \DeuxFoncteursStricts{} commutatif à transformation stricte près. Dans ce cas particulier, que l'on traite dans la section \ref{SectionCasStrict}, on peut associer à de telles données un diagramme commutatif
$$
\xymatrix{
\DeuxInt{\DeuxFoncteurTranche{w}}
\ar[rr]^{\DeuxInt{\DeuxFoncteurTranche{\sigma}}}
\ar[d]
&&\DeuxInt{\DeuxFoncteurTranche{v}}
\ar[d]
\\
\mathdeuxcat{A}
\ar[rr]_{u}
&&\mathdeuxcat{B}
}
$$
dont les flèches verticales sont des préfibrations à fibres faiblement contractiles, et donc des équivalences faibles, dans $\DeuxCat$. Il s'agit des projections associées à l'« intégration » de \DeuxFoncteursStricts{} $\DeuxFoncteurTranche{w}$ et $\DeuxFoncteurTranche{v}$ induits par les données, ce procédé d'intégration étant tout à fait analogue à celui bien connu dans le cadre $1$\nobreakdash-catégorique sous le nom de « construction de Grothendieck ». La flèche horizontale supérieure du diagramme ci-dessus, elle, est une équivalence faible en vertu des hypothèses et de résultats généraux sur l'intégration. On en déduit que $u$ est une équivalence faible. 

La notion de préfibration dans $\DeuxCat$ (définition \ref{DefPrefibration}) que nous présentons dans la section \ref{SectionRappels} s'appuie sur celle de « préadjoint » (définition \ref{DefPreadjoints}), généralisation des classiques foncteurs adjoints de $\Cat$. Dans la mesure où ces notions nous servent avant tout à l'étude des propriétés homotopiques de l'intégration (section \ref{SectionIntegration}), nous n'étudions pas ici les relations qu'elles entretiennent avec les fibrations au sens de Hermida \cite[définition 2.3]{Hermida} et de Buckley \cite[définition 2.1.6]{Buckley} ou les adjonctions dans $\DeuxCat$ respectivement. Remarquons toutefois que la notion d'intégration que nous utilisons constitue un cas très particulier de celle étudiée par Bakovi\'{c} dans \cite{Bakovic}. 

L'argument permettant de se ramener du cas « lax » au cas « strict » utilise
de façon essentielle la construction par Bénabou d'un adjoint à gauche de
l'inclusion de $\DeuxCat$ dans $\DeuxCatLax$ (catégorie dont les morphismes
sont les \DeuxFoncteursLax{}), adjonction jouissant de propriétés
homotopiques remarquables (section \ref{SectionAdjonctionBenabou}). On
montre que sa coünité est une équivalence faible terme à terme (résultat
dont une démonstration par un argument différent se trouve déjà dans
\cite{NotesDelHoyo}). Nous introduisons de plus une notion d'équivalence
faible \emph{lax} qui nous permet d'affirmer que l'unité de cette adjonction
est également une équivalence faible terme à terme. Nous vérifions ensuite
que cette propriété reste vraie dans le cas relatif (lemmes
\ref{StrictInduitEquiFaible} et \ref{LaxInduitEquiFaible}), ce qui permet de
se ramener au cas « strict » déjà traité et achève la démonstration de la
généralisation annoncée du Théorème A de Quillen (théorème
\ref{ThALaxTrancheLaxCoq}). 

Pour finir, la section \ref{SectionEqCatLoc} est consacrée à la démonstration du fait que les \deux{}catégories modélisent les types d'homotopie. Bien qu'il s'agisse d'un résultat intuitif, il ne semble avoir jamais fait l'objet d'une démonstration. Il est possible qu'un argument « à la Fritsch et Latch » permette de l'établir, mais aucun résultat publié ne va dans cette direction. Les types d'homotopie sont les objets de la « catégorie homotopique » $\Hot$, équivalente à la localisation de la catégorie des ensembles simpliciaux par les équivalences faibles entre iceux, ou encore la localisation de $Top$ par les équivalences faibles topologiques (applications continues induisant des bijections entre les $\pi_{0}$ et des isomorphismes entre les groupes d'homotopie supérieurs pour tout choix de point base), ou encore la localisation de $\Cat$ par les équivalences faibles de la structure de catégorie de modèles de Thomason \cite{Thomason} (voir aussi \cite{LFM} pour une caractérisation purement catégorique de ces équivalences faibles). Une stratégie naturelle pour établir le résultat souhaité consiste à comparer les \deux{}catégories aux modèles des types d'homotopie déjà connus, et c'est bien entendu une comparaison entre $\Cat$ et $\DeuxCat$ qui paraît la plus prometteuse. On voudrait donc relier fonctoriellement toute \deux{}catégorie à une catégorie par une équivalence faible. Malheureusement, cet espoir est illusoire, puisqu'il existe des \deux{}catégories qui ne peuvent être reliées à aucune catégorie par une équivalence faible. En effet, par exemple, si $G$ désigne un groupe abélien non trivial quelconque et $\mathdeuxcat{G}$ la \deux{}catégorie n'admettant qu'un seul objet dont la catégorie des endomorphismes est donnée par le groupoïde (à un seul élément) associé au groupe $G$, toute équivalence faible entre $\mathdeuxcat{G}$ et une catégorie se factoriserait par la catégorie ponctuelle, ce qui contredit la non-trivialité de l'homotopie de $\mathdeuxcat{G}$. En revanche, cette obstruction s'évanouit dès que l'on autorise les équivalences faibles \emph{lax}, indice de la nécessité de considérer des morphismes non stricts dans l'étude homotopique des catégories supérieures. On présente dans \cite{TheseMoi} une démonstration s'appuyant sur l'existence d'un tel remplacement \un{}catégorique. Une autre façon de faire disparaître l'obstruction consiste à autoriser les chaînes de longueur supérieure d'équivalences faibles. C'est un tel argument que nous présentons ici, reposant sur l'existence d'une chaîne de longueur $2$. On pourrait croire que cela nous dispense de considérer les morphismes non stricts. En fait, l'étude des relations entre le « monde strict » et le « monde lax » permet de dégager le résultat de façon beaucoup plus naturelle, et nous utilisons pour cela un cas particulier de la généralisation du Théorème A aux \DeuxFoncteursLax{}. Notons que, pour démontrer ce résultat précis, ceux dégagés par del Hoyo suffisent (voir la remarque \ref{RemarqueDelHoyoAuraitPu}). 

La généralisation du Théorème A de Quillen que nous présentons concerne le cas des \deux{}fonc\-teurs lax entre \deux{}catégories. Si nous avons affaibli les flèches, nous n'avons donc pas affaibli les objets. Autrement dit, nous ne nous sommes pas intéressés au cas des \DeuxFoncteursLax{} entre bicatégories. Plusieurs textes récents envisagent cependant les bicatégories du point de vue homotopique, par exemple \cite{CCG}. Nous n'avons pas ressenti le besoin de généraliser jusque-là, essentiellement pour deux raisons. Tout d'abord, en vertu d'un résultat classique de Bénabou, toute bicatégorie est biéquivalente à une \deux{}catégorie. D'autre part, la catégorie dont les objets sont les bicatégories et les morphismes les \DeuxFoncteursLax{} ne saurait être munie d'une structure de catégorie de modèles de Quillen. Il importe donc de restreindre la classe des objets ou celle des morphismes considérés. Toutefois, si l'étude homotopique des bicatégories n'entre pas dans notre propos, elle a bien entendu son intérêt propre. On en trouvera quelques motivations dans la littérature sur le sujet. La mise à jour du présent article nous permet de plus de mentionner le récent travail de Calvo, Cegarra et Heredia \cite{CCH} généralisant le Théorème B de Quillen aux \DeuxFoncteursLax{} entre bicatégories.

Le présent travail s'inscrit dans un programme d'étude homotopique systématique des \deux{}ca\-té\-go\-ries dont la notion de \emph{localisateur fondamental} de $\DeuxCat$ constitue l'un des pivots. Il s'agit de l'analogue, pour la catégorie $\DeuxCat$, de la notion de localisateur fondamental définie pour $\Cat$ par Grothendieck dans \emph{Pursuing Stacks} \cite{Poursuite, THG}. Les résultats présentés ici restent valables dans le cas d'un localisateur fondamental de $\DeuxCat$ arbitraire. Le travail de Cisinski \cite{PMTH} fournit de très nombreux exemples utiles de localisateurs fondamentaux de $\Cat$, et donc de localisateurs fondamentaux de $\DeuxCat$. Cela permet par exemple d'étendre les résultats du présent texte au cas des types d'homotopie tronqués. Pour des raisons de commodité, nous avons toutefois choisi de nous limiter ici au cas des équivalences faibles « classiques » et de présenter la notion générale dans notre thèse \cite{TheseMoi} et l'article \cite{ArticleLocFondMoi}, auquel fera suite l'article \cite{Ara} dans lequel Dimitri Ara construit des structures de catégorie de modèles sur $\DeuxCat$ pour essentiellement tout localisateur fondamental et dans lequel il étudie les relations remarquables entre ces structures et leurs analogues sur $\Cat$ et sur la catégorie des ensembles simpliciaux.

C'est la démonstration par Georges Maltsiniotis \cite{CourrielGeorges} du Théorème A relatif « non-commutatif » en dimension $1$ qui nous a convaincu de la correction de l'énoncé \deux{}di\-men\-sion\-nel analogue et suggéré la stratégie de démonstration\footnote{On pourra noter la similitude avec l'approche adoptée dans \cite{CCH}.} à partir du cas « commutatif ». Antonio Cegarra, répondant à nos questions relatives aux articles \cite{BC} et \cite{Cegarra}, nous a signalé la thèse \cite{TheseDelHoyo} et l'auteur de cette dernière, Matias del Hoyo, nous a communiqué ses notes \cite{NotesDelHoyo}, qui restent inédites mais contiennent la première démonstration dont nous ayons connaissance d'un analogue du Théorème A de Quillen pour les \deux{}foncteurs lax généraux. Nos résultats s'appuient sur ceux de Bullejos, Cegarra et del Hoyo. La version révisée du présent texte reflète des suggestions faites par Steve Lack après sa lecture de \cite{TheseMoi} : c'est grâce à lui que nous avons découvert que certaines notions que nous utilisons se trouvaient déjà traitées dans la littérature, notamment dans \cite{BettiPower}. Nous sommes redevable à Jean Bénabou de plusieurs explications éclairantes et d'encouragements chaleureux que nous espérons n'avoir pas trop déçus. Enfin, pour de nombreuses discussions qui nous ont été très utiles et son soutien, nous exprimons notre reconnaissance à Dimitri Ara.

\section{Conventions, rappels et préliminaires}\label{SectionRappels}

\begin{paragr}
On suppose connue la notion de \deux{}catégorie, et l'on ne rappelle pas non plus celle de \DeuxFoncteurStrict{}. À l'exception des catégories (ou \deux{}catégories) dont les objets sont les petites catégories (ou les petites \deux{}catégories), toutes les catégories et \deux{}catégories considérées seront petites, et nous ferons l'économie de l'adjectif, qui sera souvent sous-entendu. La composée des \un{}cellules sera notée par la juxtaposition (par exemple $f'f$), de même que la composition verticale des \deux{}cellules (par exemple $\beta \alpha$). On notera la composition horizontale des \deux{}cellules par « $\CompDeuxZero{}$ » (par exemple $\alpha' \CompDeuxZero \alpha$). On commettra l'abus sans conséquence fâcheuse consistant à confondre une \un{}cellule avec son identité. Ainsi, par exemple, pour toute \un{}cellule $f$ et toute \deux{}cellule $\alpha$ telles que la composée $1_{f} \CompDeuxZero \alpha$ fasse sens, on notera souvent cette composée $f \CompDeuxZero \alpha$. En symboles, on notera les \un{}cellules par des flèches simples « $\to$ » et les \deux{}cellules par des flèches doubles « $\Rightarrow$ ». Pour deux objets $a$ et $a'$ d'une \deux{}catégorie $\mathdeuxcat{A}$, on notera $\CatHom{\mathdeuxcat{A}}{a}{a'}$ la catégorie dont les objets sont les \un{}cellules de $a$ vers $a'$ et dont les morphismes sont les \deux{}cellules entre icelles dans $\mathdeuxcat{A}$. Pour toute \deux{}catégorie $\mathdeuxcat{A}$, on notera $\DeuxCatUnOp{\mathdeuxcat{A}}$ la \deux{}catégorie obtenue à partir de $\mathdeuxcat{A}$ « en inversant le sens des \un{}cellules », $\DeuxCatDeuxOp{\mathdeuxcat{A}}$ la \deux{}catégorie obtenue à partir de $\mathdeuxcat{A}$ « en inversant le sens des \deux{}cellules » et $\DeuxCatToutOp{\mathdeuxcat{A}}$ la \deux{}catégorie obtenue à partir de $\mathdeuxcat{A}$ « en inversant le sens des \un{}cellules et des \deux{}cellules ». 
\end{paragr}

\begin{paragr}
Un \deux{}\emph{foncteur lax} $u$ d'une \deux{}catégorie $\mathdeuxcat{A}$ vers une \deux{}catégorie $\mathdeuxcat{B}$ correspond à la donnée d'un objet $u(a)$ de $\mathdeuxcat{B}$ pour tout objet $a$ de $\mathdeuxcat{A}$, d'une \un{}cellule $u(f)$ de $u(a)$ vers $u(a')$ dans $\mathdeuxcat{B}$ pour toute \un{}cellule $f$ de $a$ vers $a'$ dans $\mathdeuxcat{A}$, d'une \deux{}cellule $u(\alpha)$ de $u(f)$ vers $u(g)$ dans $\mathdeuxcat{B}$ pour toute \deux{}cellule $\alpha$ de $f$ vers $g$ dans $\mathdeuxcat{A}$, d'une \deux{}cellule $u_{a}$ de $1_{u(a)}$ vers $u(1_{a})$ dans $\mathdeuxcat{B}$ pour tout objet $a$ de $\mathdeuxcat{A}$ et, pour tout couple $(f', f)$ de \un{}cellules de $\mathdeuxcat{A}$ telles que la composée $f'f$ fasse sens, d'une \deux{}cellule $u_{f', f}$ de $u(f') u(f)$ vers $u(f'f)$ dans $\mathdeuxcat{B}$, ces données vérifiant les conditions de cohérence bien connues. On appellera les \deux{}cellules $u_{a}$ et $u_{f', f}$ les \emph{2-cellules structurales de} $u$ associées à $a$ et au couple $(f, f')$ (ou $(f', f)$) respectivement. Un \DeuxFoncteurLax{} est donc un \DeuxFoncteurStrict{} si et seulement si toutes ses \deux{}cellules structurales sont des identités. On appellera \deux{}\emph{foncteur colax} de $\mathdeuxcat{A}$ vers $\mathdeuxcat{B}$ un \DeuxFoncteurLax{} de $\DeuxCatDeuxOp{\mathdeuxcat{A}}$ vers $\DeuxCatDeuxOp{\mathdeuxcat{B}}$. Nous attirons l'attention sur le fait que la terminologie actuellement dominante utilise « oplax » pour ce que nous appelons « colax ». Le choix que nous avons fait procède du désir de se conformer à la convention suivant laquelle les préfixes « op- » et « co- » indiquent un changement de sens des \un{}cellules et des \deux{}cellules respectivement. Tout \DeuxFoncteurLax{} $u : \mathdeuxcat{A} \to \mathdeuxcat{B}$ induit un \DeuxFoncteurLax{} $\DeuxFoncUnOp{u} : \DeuxCatUnOp{\mathdeuxcat{A}} \to \DeuxCatUnOp{\mathdeuxcat{B}}$ ainsi que des \DeuxFoncteursCoLax{} $\DeuxFoncDeuxOp{u} : \DeuxCatDeuxOp{\mathdeuxcat{A}} \to \DeuxCatDeuxOp{\mathdeuxcat{B}}$ et $\DeuxFoncToutOp{u} : \DeuxCatToutOp{\mathdeuxcat{A}} \to \DeuxCatToutOp{\mathdeuxcat{B}}$, et réciproquement. 
\end{paragr}

\begin{df}\label{DefDeuxCellStruct}
Soit $u$ un \DeuxFoncteurLax{} de source $\mathdeuxcat{A}$, $n \geq 3$ un entier et $(f_{1}, \dots, f_{n})$ un $n$\nobreakdash-uplet de \un{}cellules de $\mathdeuxcat{A}$ telles que la composée $f_{n} \dots f_{1}$ fasse sens. On définit inductivement une \deux{}cellule $u_{f_{n}, \dots, f_{1}}$ de $u(f_{n}) \dots u(f_{1})$ vers $u(f_{n} \dots f_{1})$ par 
$$
u_{f_{n}, \dots, f_{1}} = u_{f_{n}, f_{n-1} \dots f_{1}} \left(u(f_{n}) \CompDeuxZero u_{f_{n-1}, \dots, f_{1}}\right)
$$
On appellera cette \deux{}cellule la \deux{}\emph{cellule structurale de} $u$ \emph{associée au n\nobreakdash-uplet} $(f_{1}, \dots, f_{n})$ (ou $(f_{n}, \dots, f_{1})$). On la notera parfois $u_{(f)}$ si le contexte rend clair cette notation. Si cela n'entraîne pas d'ambiguïté, cette \deux{}cellule pourra aussi se trouver notée $u_{x_{0} \to \dots \to x_{n}}$, étant entendu que $x_{i}$ est le but de $f_{i}$. 
\end{df}

\begin{rem}
Ces \deux{}cellules structurales vérifient une « condition de cocycle généralisée » que nous n'explicitons pas mais que, suivant l'usage, nous utiliserons sans le signaler. On renvoie le lecteur intéressé à \cite[proposition 5.1.4 (p. I-47) et théorème 5.2.4 (p. I-49)]{TheseBenabou} pour un résultat beaucoup plus général. 
\end{rem}

\begin{paragr}
Étant donné deux \DeuxFoncteursLax{} (ou deux \DeuxFoncteursCoLax{}) $u$ et $v$ de $\mathdeuxcat{A}$ vers $\mathdeuxcat{B}$, une \emph{transformation} $\sigma$ de $u$ vers $v$ correspond à la donnée d'une \un{}cellule $\sigma_{a} : u(a) \to v(a)$ dans $\mathdeuxcat{B}$ pour tout objet $a$ de $\mathdeuxcat{A}$ et d'une \deux{}cellule $\sigma_{f} : \sigma_{a'} u(f) \Rightarrow v(f) \sigma_{a}$ dans $\mathdeuxcat{B}$ pour toute \un{}cellule $f : a \to a'$ dans $\mathdeuxcat{A}$, ces données vérifiant les conditions de cohérence bien connues. Avec les mêmes données, une \emph{optransformation} de $u$ vers $v$ est une \DeuxTransformationLax{} de $\DeuxFoncUnOp{v}$ vers $\DeuxFoncUnOp{u}$. De façon plus explicite, cela revient à se donner une \un{}cellule $\sigma_{a} : u(a) \to v(a)$ dans $\mathdeuxcat{B}$ pour tout objet $a$ de $\mathdeuxcat{A}$ et une \deux{}cellule $\sigma_{f} : v(f) \sigma_{a} \Rightarrow \sigma_{a'} u(f)$ dans $\mathdeuxcat{B}$ pour toute \un{}cellule $f : a \to a'$ dans $\mathdeuxcat{A}$, ces données vérifiant les conditions de cohérence aussi bien connues que les précédentes. Nous attirons l'attention sur le fait que la terminologie présente dans la littérature récente privilégie les termes « transformation lax » et « transformation oplax », sans qu'il y ait d'accord quant à la convention relative au choix de ce qui est « lax » et ce qui est « oplax ». En revanche, nous nous trouverons en accord avec la terminologie ambiante en ce que nous appellerons \emph{transformation stricte} de $u$ vers $v$ une \DeuxTransformationLax{} (ou, ce qui revient au même dans ce cas précis, une \DeuxTransformationOpLax{}) $\sigma$ de $u$ vers $v$ telle que la \deux{}cellule $\sigma_{f}$ soit une identité pour toute \un{}cellule $f$ de $\mathdeuxcat{A}$. 
\end{paragr}

\begin{paragr}
On notera $\DeuxCat$ la catégorie dont les objets sont les \deux{}catégories et dont les morphismes sont les \DeuxFoncteursStricts{}. C'est une sous-catégorie (non pleine) de la catégorie $\DeuxCatLax$, dont les objets sont les \deux{}catégories et dont les morphismes sont les \DeuxFoncteursLax{}. On notera $\DeuxCatDeuxCat$ la \deux{}catégorie dont la catégorie sous-jacente est $\DeuxCat$ et dont les \deux{}cellules sont les \DeuxTransformationsStrictes{} (c'est en fait la \deux{}catégorie sous-jacente à une 3\nobreakdash-catégorie dont les 3\nobreakdash-cellules sont les modifications).  
\end{paragr}

\begin{paragr}
Soient $u : \mathdeuxcat{A} \to \mathdeuxcat{B}$ un \DeuxFoncteurLax{} et $b$ un objet de $\mathdeuxcat{B}$. On note $\TrancheLax{\mathdeuxcat{A}}{u}{b}$ la \deux{}catégorie définie comme suit. Ses objets sont les couples $(a, p)$ avec $a$ un objet de $\mathdeuxcat{A}$ et $p$ une \un{}cellule de $u(a)$ vers $b$ dans $\mathdeuxcat{B}$. Les \un{}cellules de $(a,p)$ vers $(a',p')$ sont les couples $(f, \alpha)$ avec $f$ une \un{}cellule de $a$ vers $a'$ dans $\mathdeuxcat{A}$ et $\alpha$ une \deux{}cellule de $p$ vers $p' u(f)$ dans $\mathdeuxcat{B}$. Les \deux{}cellules de $(f, \alpha)$ vers $(g, \beta)$ sont les \deux{}cellules $\gamma : f \Rightarrow g$ dans $\mathdeuxcat{A}$ telles que $(p' \CompDeuxZero u(\gamma)) \alpha = \beta$. Les diverses unités et compositions sont héritées de celles de $\mathdeuxcat{A}$ (autrement dit, elles sont « évidentes »). 
\end{paragr}

\begin{rem}
Comme nous l'a fait remarquer Steve Lack, cette définition n'est qu'un cas particulier de celle de \emph{\deux{}catégorie comma}. Il en va de même des variantes duales qui suivent. Les propriétés de fonctorialité que nous utilisons plus loin peuvent également s'énoncer dans un cadre plus général. Nous traitons avec davantage de détails cette question dans \cite[sections 1.4, 1.5 et 1.9]{TheseMoi}.
\end{rem}

\begin{paragr}
Sous les mêmes hypothèses, on définit une \deux{}catégorie $\OpTrancheLax{\mathdeuxcat{A}}{u}{b}$ par la formule 
$$
\OpTrancheLax{\mathdeuxcat{A}}{u}{b} = \DeuxCatUnOp{\left(\TrancheLax{(\DeuxCatUnOp{\mathdeuxcat{A}})}{\DeuxFoncUnOp{u}}{b}\right)}
$$
En particulier, les objets, \un{}cellules et \deux{}cellules de $\OpTrancheLax{\mathdeuxcat{A}}{u}{b}$ se décrivent comme suit. Les objets sont les couples $(a, p : b \to u(a))$, où $a$ est un objet de $\mathdeuxcat{A}$ et $p$ une \un{}cellule de $\mathdeuxcat{B}$. Les \un{}cellules de $(a,p)$ vers $(a',p')$ sont les couples $(f : a \to a', \alpha : p' \Rightarrow u(f) p)$, où $f$ est une \un{}cellule de $\mathdeuxcat{A}$ et $\alpha$ est une \deux{}cellule de $\mathdeuxcat{B}$. Les \deux{}cellules de $(f, \alpha)$ vers $(f', \alpha')$ sont les \deux{}cellules $\beta : f \Rightarrow f'$ dans $\mathdeuxcat{A}$ telles que $(u(\beta) \CompDeuxZero p) \CompDeuxUn \alpha = \alpha'$.
\end{paragr}

\begin{paragr}
Dualement, si $u : \mathdeuxcat{A} \to \mathdeuxcat{B}$ est un \DeuxFoncteurCoLax{} et $b$ un objet de $\mathdeuxcat{B}$, on définit une \deux{}catégorie $\TrancheCoLax{\mathdeuxcat{A}}{u}{b}$ par la formule 
$$
\TrancheCoLax{\mathdeuxcat{A}}{u}{b} = \DeuxCatDeuxOp{\left(\TrancheLax{(\DeuxCatDeuxOp{\mathdeuxcat{A}})}{\DeuxFoncDeuxOp{u}}{b}\right)}
$$
En particulier, les objets, \un{}cellules et \deux{}cellules de $\TrancheCoLax{\mathdeuxcat{A}}{u}{b}$ se décrivent comme suit. Les objets sont les couples $(a, p : u(a) \to b)$, où $a$ est un objet de $\mathdeuxcat{A}$ et $p$ une \un{}cellule de $\mathdeuxcat{B}$. Les \un{}cellules de $(a,p)$ vers $(a',p')$ sont les couples $(f : a \to a', \alpha : p' u(f) \Rightarrow p)$, où $f$ est une \un{}cellule de $\mathdeuxcat{A}$ et $\alpha$ est une \deux{}cellule de $\mathdeuxcat{B}$. Les \deux{}cellules de $(f, \alpha)$ vers $(f', \alpha')$ sont les \deux{}cellules $\beta : f \Rightarrow f'$ dans $\mathdeuxcat{A}$ telles que $\alpha' \CompDeuxUn (p' \CompDeuxZero u(\beta)) = \alpha$.
\end{paragr}

\begin{paragr}
Sous les mêmes hypothèses, on définit une \deux{}catégorie $\OpTrancheCoLax{\mathdeuxcat{A}}{u}{b}$ par la formule 
$$
\OpTrancheCoLax{\mathdeuxcat{A}}{u}{b} = \DeuxCatUnOp{\left(\TrancheCoLax{(\DeuxCatUnOp{\mathdeuxcat{A}})}{\DeuxFoncUnOp{u}}{b}\right)}
$$
En particulier, les objets, \un{}cellules et \deux{}cellules de $\OpTrancheCoLax{\mathdeuxcat{A}}{u}{b}$ se décrivent comme suit. Les objets sont les couples $(a, p : b \to u(a))$, où $a$ est un objet de $\mathdeuxcat{A}$ et $p$ une \un{}cellule de $\mathdeuxcat{B}$. Les \un{}cellules de $(a,p)$ vers $(a',p')$ sont les couples $(f : a \to a', \alpha : u(f) p \Rightarrow p')$, où $f$ est une \un{}cellule de $\mathdeuxcat{A}$ et $\alpha$ est une \deux{}cellule de $\mathdeuxcat{B}$. Les \deux{}cellules de $(f, \alpha)$ vers $(f', \alpha')$ sont les \deux{}cellules $\beta : f \Rightarrow f'$ dans $\mathdeuxcat{A}$ telles que $\alpha' \CompDeuxUn (u(\beta) \CompDeuxZero p) = \alpha$.
\end{paragr}

Si $u$ est une identité, on ne le fera pas figurer dans ces notations.

\begin{paragr}
Soit 
$$
\xymatrix{
\mathdeuxcat{A}  
\ar[rr]^{u} 
\ar[dr]_{w} 
&&\mathdeuxcat{B}
\dtwocell<\omit>{<7.3>\sigma} 
\ar[dl]^{v}
\\ 
& 
\mathdeuxcat{C}
&{}
}
$$
un diagramme dans lequel $u$, $v$ et $w$ sont des \DeuxFoncteursLax{} et $\sigma$ est une \DeuxTransformationOpLax{}. Pour tout objet $c$ de $\mathdeuxcat{C}$, ces données induisent un \DeuxFoncteurLax{} $\DeuxFoncTrancheLaxCoq{u}{\sigma}{c}$ de $\TrancheLax{\mathdeuxcat{A}}{w}{c}$ vers $\TrancheLax{\mathdeuxcat{B}}{v}{c}$ défini par $(\DeuxFoncTrancheLaxCoq{u}{\sigma}{c}) (a, p) = (u(a), p \sigma_{a})$ et $(\DeuxFoncTrancheLaxCoq{u}{\sigma}{c})_{(a,p)} = u_{a}$ pour tout objet $(a,p)$ de $\TrancheLax{\mathdeuxcat{A}}{w}{c}$, $(\DeuxFoncTrancheLaxCoq{u}{\sigma}{c}) (f, \alpha) = (u(f), (p' \CompDeuxZero \sigma_{f}) (\alpha \CompDeuxZero \sigma_{a}))$ pour toute \un{}cellule $(f, \alpha)$ de $\TrancheLax{\mathdeuxcat{A}}{w}{c}$, $(\DeuxFoncTrancheLaxCoq{u}{\sigma}{c}) (\beta) = u(\beta)$ pour toute \deux{}cellule $\beta$ de $\TrancheLax{\mathdeuxcat{A}}{w}{c}$ et $(\DeuxFoncTrancheLaxCoq{u}{\sigma}{c})_{(f', \alpha'), (f, \alpha)} = u_{f', f}$ pour tout couple $((f', \alpha'), (f, \alpha))$ de \un{}cellules de $\TrancheLax{\mathdeuxcat{A}}{w}{c}$ telles que la composée $(f', \alpha') (f, \alpha)$ fasse sens. Si $\sigma$ est une identité, on notera $\DeuxFoncTrancheLaxCoq{u}{}{c}$ le \DeuxFoncteurLax{} induit par les données ci-dessus. Notons que, dans le cas général, $\DeuxFoncTrancheLaxCoq{u}{\sigma}{c}$ est un \DeuxFoncteurStrict{} dès qu'il en est de même de $u$. 
\end{paragr}

L'objet principal de cet article est de montrer que si, partant des données ci-dessus, $\DeuxFoncTrancheLaxCoq{u}{\sigma}{c}$ est une équivalence faible pour tout objet $c$ de $\mathdeuxcat{C}$, alors $u$ est une équivalence faible. Pour donner un sens à cet énoncé, il nous faut introduire la notion d'équivalence faible. 

\begin{paragr}
Pour tout entier $m \geq 0$, on note $[m]$ la catégorie associée à l'ensemble $\{0, \dots, m \}$ ordonné par l'ordre naturel sur l'ensemble des entiers. (C'est donc, assez malheureusement, la catégorie associée à l'ordinal $m+1$, ce qui résulte des conventions simpliciales.) Pour tout \DeuxFoncteurLax{} $x : [m] \to \mathdeuxcat{A}$, on notera $x_{i}$ l'image par $x$ de l'objet $i$ de $[m]$, $x_{j,i}$ l'image  par $x$ de l'unique morphisme de $i$ vers $j$ dans $[m]$, pour tout couple $(i,j)$ vérifiant $0 \leq i \leq j \leq m$, et $x_{k, j, i}$ la \deux{}cellule structurale de $x$ associée au couple de \un{}cellules composables $(i \to j, j \to k)$ de $[m]$, pour tout triplet $(i, j, k)$ vérifiant $0 \leq i \leq j \leq k \leq m$. Pour éviter les confusions, on pourra noter $(x)_{i}$ la \deux{}cellule structurale $1_{x_{i}} \Rightarrow x(1_{i})$.
\end{paragr}

\begin{paragr}
Pour toute \deux{}catégorie $\mathdeuxcat{A}$, on note $\NerfLax{\mathdeuxcat{A}}$ l'ensemble simplicial défini par 
$$
(\NerfLax{\mathdeuxcat{A}})_{m} = \EnsHom{\DeuxCatLax}{[m]}{\mathdeuxcat{A}}
$$ 
(l'ensemble des \DeuxFoncteursLax{} de $[m]$ vers $\mathdeuxcat{A}$) pour tout entier $m \geq 0$ et dont les faces et dégénérescences sont définies de façon « évidente ». On vérifie la fonctorialité de cette construction, et l'on associe, à tout \DeuxFoncteurLax{} $u : \mathdeuxcat{A} \to \mathdeuxcat{B}$, un morphisme d'ensembles simpliciaux $\NerfLax{(u)} : \NerfLax{\mathdeuxcat{A}} \to \NerfLax{\mathdeuxcat{B}}$. L'ensemble simplicial $\NerfLax{\mathdeuxcat{A}}$ s'appelle le \emph{nerf} (ou \emph{nerf lax}, s'il convient de distinguer) de $\mathdeuxcat{A}$ et le morphisme simplicial $\NerfLax{(u)}$ le \emph{nerf} (ou \emph{nerf lax}, s'il convient de distinguer) de $u$. En vertu des résultats de \cite{CCG}, cette définition du nerf est équivalente, du point de vue homotopique, à celle fournie par le nerf défini par Street dans \cite{StreetAOS}, dont le lecteur est peut-être plus familier mais qui n'est pas fonctoriel sur les \DeuxFoncteursLax{} en général. On dira donc en particulier qu'un \DeuxFoncteurLax{} est une équivalence faible si son nerf est une équivalence faible simpliciale. Notons qu'il est possible de caractériser la classe des équivalences faibles de façon purement \deux{}catégorique, sans nullement faire appel aux ensembles simpliciaux ou aux espaces topologiques (voir \cite{TheseMoi} ou \cite{ArticleLocFondMoi} pour un traitement de cette question).
\end{paragr}

\begin{paragr}
Une propriété essentielle de la classe des équivalences faibles est qu'elle est \emph{faiblement saturée}\footnote{Elle est même \emph{fortement} saturée, mais nous n'utiliserons pas cette propriété.}, c'est-à-dire qu'elle vérifie les points suivants.

\begin{itemize}
\item[FS1] Les identités sont des équivalences faibles.
\item[FS2] Si deux des trois flèches d'un triangle commutatif sont des équivalences faibles, alors la troisième en est une.
\item[FS3] Si $i : \mathdeuxcat{A} \to \mathdeuxcat{B}$ et $r :  \mathdeuxcat{B} \to  \mathdeuxcat{A}$ vérifient $ri = 1_{\mathdeuxcat{A}}$ et si $ir$ est une équivalence faible, alors il en est de même de $r$ (et donc aussi de $i$ en vertu de ce qui précède). 
\end{itemize}

En particulier, tout isomorphisme est une équivalence faible. Suivant l'usage, on appellera la propriété FS2 propriété « de $2$ sur $3$ ». On notera $\DeuxLocFond{W}$ la classe des équivalences faibles qui sont des \DeuxFoncteursStricts{} (les « équivalences faibles strictes ») et $\DeuxLocFondLaxInduit{W}$ la classe des équivalences faibles qui sont des \DeuxFoncteursLax{} (les « équivalences faibles lax »).
\end{paragr}

\begin{paragr}
On notera $\DeuxCatPonct$ la \deux{}catégorie ponctuelle, qui ne possède qu'un seul objet, qu'une seule \un{}cellule et qu'une seule \deux{}cellule. On se permettra de la confondre avec la catégorie ponctuelle (bien que cette dernière n'ait, à proprement parler, aucune \deux{}cellule). 
\end{paragr}

\begin{df}
On dira qu'une \deux{}catégorie $\mathdeuxcat{A}$ est \emph{asphérique}\footnote{Plutôt que « faiblement contractile », utilisé dans l'introduction. Nous adoptons donc la terminologie de Grothendieck.} si le \DeuxFoncteurStrict{} canonique $\mathdeuxcat{A} \to \DeuxCatPonct$ est une équivalence faible.
\end{df}

\begin{paragr}\label{AsphOpCo}
On utilisera le fait élémentaire qu'une \deux{}catégorie $\mathdeuxcat{A}$ est asphérique si et seulement si $\DeuxCatUnOp{\mathdeuxcat{A}}$ l'est, ou encore si et seulement si $\DeuxCatDeuxOp{\mathdeuxcat{A}}$ l'est, la réalisation géométrique « ne tenant pas compte des orientations ».
\end{paragr}

\begin{df}\label{DefHomotopie}
Étant donné des \DeuxFoncteursLax{} $u$ et $v$ de $\mathdeuxcat{A}$ vers $\mathdeuxcat{B}$, une \emph{homotopie élémentaire de} $u$ \emph{vers} $v$ est un \DeuxFoncteurLax{} $h : [1] \times \mathdeuxcat{A} \to \mathdeuxcat{B}$ tel que le diagramme\footnote{Dans lequel on commet l'abus d'identifier $\mathdeuxcat{A}$ à $\DeuxCatPonct{} \times \mathdeuxcat{A}$.} 
$$
\xymatrix{
&[1] \times \mathdeuxcat{A}
\ar[dd]^{h}
\\
\mathdeuxcat{A}
\ar[ur]^{(\{0\}, 1_{\mathdeuxcat{A}})}
\ar[dr]_{u}
&&\mathdeuxcat{A}
\ar[ul]_{(\{1\}, 1_{\mathdeuxcat{A}})}
\ar[dl]^{v}
\\
&\mathdeuxcat{B}
}
$$
soit commutatif. 
\end{df}

\begin{rem}
Dans ce cas, on pourra dire que $u$ est \emph{élémentairement homotope à} $v$, mais l'on prendra garde au fait que la relation « être élémentairement homotope à » n'est en règle générale ni symétrique, ni transitive. 
\end{rem}

\begin{lemme}\label{HomotopieW}
S'il existe une homotopie élémentaire d'un \DeuxFoncteurLax{} $u$ vers un \deux{}fonc\-teur lax $v$, alors $u$ est une équivalence faible si et seulement si $v$ en est une.
\end{lemme}

\begin{proof}
Notons $\mathdeuxcat{A}$ la source de $u$ et $v$. Comme les deux inclusions canoniques $(\{0\}, 1_{\mathdeuxcat{A}})$ et $(\{1\}, 1_{\mathdeuxcat{A}})$ de $\mathdeuxcat{A}$ vers $[1] \times \mathdeuxcat{A}$ sont des équivalences faibles, le résultat découle de deux applications consécutives de la propriété « de $2$ sur $3$ ».
\end{proof}

\begin{lemme}\label{DeuxTransFoncLax}
Soient $u, v : \mathdeuxcat{A} \to \mathdeuxcat{B}$ deux \DeuxFoncteursLax{} parallèles. S'il existe une \DeuxTransformationLax{} ou une \DeuxTransformationOpLax{} de $u$ vers $v$, alors il existe une homotopie élémentaire de $u$ vers $v$.
\end{lemme}

\begin{proof}
Une fois l'énoncé dégagé, la démonstration ne présente aucune difficulté. Voir \cite[démonstration de la proposition 7.1.(ii)]{CCG}. Les calculs de cohérence sont détaillés dans \cite{TheseMoi}.
\end{proof}

%

\begin{paragr}
La réalisation géométrique du nerf d'une catégorie admettant un objet initial ou un objet final est un espace faiblement contractile \cite[corollaire 2, page 8]{QuillenK}. La définition \ref{DefOF2} introduit une notion analogue, du point de vue homotopique, à celle d'objet initial ou final, dans le contexte \deux{}catégorique. 
\end{paragr}

\begin{df}\label{DefOF2}
On dira qu'un objet $z$ d'une \deux{}catégorie $\mathdeuxcat{A}$ \emph{admet un objet final} si, pour tout objet $a$ de $\mathdeuxcat{A}$, la catégorie $\CatHom{\mathdeuxcat{A}}{a}{z}$ admet un objet final. On dira qu'il \emph{admet un objet initial} s'il admet un objet final dans $\DeuxCatDeuxOp{\mathdeuxcat{A}}$, autrement dit si, pour tout objet $a$ de $\mathdeuxcat{A}$, la catégorie $\CatHom{\mathdeuxcat{A}}{a}{z}$ admet un objet initial. 
\end{df}

\begin{exemple}
Un objet de la \deux{}catégorie $\Cat$ admet un objet final (\emph{resp.} initial) en ce sens si et seulement si c'est une catégorie admettant un objet final (\emph{resp.} initial) au sens usuel ; c'est la raison de l'adoption de cette terminologie.
\end{exemple}

\begin{rem}\label{RemarqueJay}
Nous devons à la consultation de \cite{BettiPower} d'avoir appris que la notion d'objet admettant un objet final se trouvait déjà mise en valeur par Jay, dans \cite{Jay}, sous le nom d'\emph{objet final local}, dans un contexte bicatégorique. 
\end{rem}

\begin{exemple}\label{ExemplesOF}
Si $\mathdeuxcat{A}$ est une \deux{}catégorie et $a$ un objet de $\mathdeuxcat{A}$, alors $\TrancheCoLax{\mathdeuxcat{A}}{}{a}$ (\emph{resp.} $\TrancheLax{\mathdeuxcat{A}}{}{a}$) admet un objet admettant un objet final (\emph{resp.} initial). 
\end{exemple}

\begin{lemme}\label{OFAspherique}
Une \deux{}catégorie admettant un objet admettant un objet final (\emph{resp.} initial) est asphérique. 
\end{lemme}

\begin{proof}
Il suffit de vérifier la première assertion (voir la remarque \ref{AsphOpCo}). Soit donc $z$ un objet d'une \deux{}catégorie $\mathdeuxcat{A}$ tel que, pour tout objet $a$ de $\mathdeuxcat{A}$, la catégorie $\CatHom{\mathdeuxcat{A}}{a}{z}$ admette un objet final, que l'on notera $p_{a}$. Pour toute \un{}cellule $f : a \to z$ de $\mathdeuxcat{A}$, l'on notera $\varphi_{f}$ l'unique \deux{}cellule de $f$ vers $p_{a}$. On définit un \deux{}endofoncteur constant (c'est-à-dire se factorisant par la \deux{}catégorie ponctuelle $\DeuxCatPonct$) $Z$ de $\mathdeuxcat{A}$ par les formules
$$
\begin{aligned}
Z : \mathdeuxcat{A} &\to \mathdeuxcat{A}
\\
a &\mapsto z
\\
f &\mapsto 1_{z}
\\
\alpha &\mapsto 1_{1_{z}}
\end{aligned}
$$
On pose alors $\sigma_{a} = p_{a}$ pour tout objet $a$ et $\sigma_{f} = \varphi_{p_{a'} f}$ pour toute \un{}cellule $f : a \to a'$ de $\mathdeuxcat{A}$. Cela définit une \DeuxTransformationLax{} de $1_{\mathdeuxcat{A}}$ vers $Z$. En vertu du lemme \ref{DeuxTransFoncLax}, il existe une homotopie élémentaire de $1_{\mathdeuxcat{A}}$ vers $Z$. Le lemme \ref{HomotopieW} permet donc d'affirmer que $Z$ est une équivalence faible.Le \DeuxFoncteurStrict{} canonique $\mathdeuxcat{A} \to \DeuxCatPonct$ est donc une équivalence faible.
\end{proof}

\begin{corollaire}\label{TranchesAspheriques}
Soit $a$ un objet d'une \deux{}catégorie $\mathdeuxcat{A}$. Les \deux{}catégories $\TrancheLax{\mathdeuxcat{A}}{}{a}$, $\TrancheCoLax{\mathdeuxcat{A}}{}{a}$, $\OpTrancheLax{\mathdeuxcat{A}}{}{a}$ et $\OpTrancheCoLax{\mathdeuxcat{A}}{}{a}$ sont asphériques.
\end{corollaire}

\begin{proof}
C'est une conséquence immédiate de l'exemple \ref{ExemplesOF} et du lemme \ref{OFAspherique}.
\end{proof}

Dans un souci de cohérence interne, nous voulons maintenant exposer une démonstration du cas relatif du Théorème A de Quillen pour les \DeuxFoncteursStricts{} (théorème \ref{ThABC}), résultat dont le cas absolu se trouve montré par Bullejos et Cegarra dans \cite{BC} et généralisé par Cegarra dans \cite{Cegarra}. Nous utilisons les résultats de ce dernier texte, sans nullement prétendre faire œuvre originale : l'exercice consiste à traduire les arguments de \cite{BC} dans le langage de \cite{Cegarra}, en passant au cas relatif. La manœuvre est suffisamment non triviale pour que nous considérions utile d'en exposer les détails. À l'instar des auteurs des articles \cite{BC} et \cite{Cegarra}, nous utiliserons la notion de \deux{}foncteur lax normalisé ainsi que celle du nerf qui lui est associé. Du point de vue homotopique, cela ne change rien, comme on va le voir, mais les calculs s'en trouvent simplifiés. Le lecteur qui souhaiterait travailler avec le seul nerf lax saura modifier les démonstrations de la façon qui s'impose.

On dira qu'un \DeuxFoncteurLax{} $u : \mathdeuxcat{A} \to \mathdeuxcat{B}$ est \emph{normalisé} si, pour tout objet $a$ de $\mathdeuxcat{A}$, on a $u(1_{a}) = 1_{u(a)}$ et si, pour toute \un{}cellule $f : a \to a'$ dans $\mathdeuxcat{A}$, on a $u_{1_{a'}, f} = u_{f, 1_{a}} = 1_{u(f)}$. On vérifie sans difficulté que cela permet de définir une catégorie $\DeuxCatLaxNor$ dont les objets sont les \deux{}catégories et les morphismes les \deux{}foncteurs lax normalisés. 

Pour toute \deux{}catégorie $\mathdeuxcat{A}$, on note $\NerfLaxNor{\mathdeuxcat{A}}$ l'ensemble simplicial défini par 
$$
(\NerfLaxNor{\mathdeuxcat{A}})_{m} = \EnsHom{\DeuxCatLaxNor}{[m]}{\mathdeuxcat{A}}
$$ 
(l'ensemble des \DeuxFoncteursLax{} normalisés de $[m]$ vers $\mathdeuxcat{A}$) pour tout entier $m \geq 0$ et dont les faces et dégénérescences sont définies de façon « évidente ». On vérifie la fonctorialité de cette construction, et l'on associe, à tout \DeuxFoncteurLax{} normalisé $u : \mathdeuxcat{A} \to \mathdeuxcat{B}$, un morphisme d'ensembles simpliciaux $\NerfLaxNor{u} : \NerfLaxNor{\mathdeuxcat{A}} \to \NerfLaxNor{\mathdeuxcat{B}}$. On désignera le morphisme simplicial $\NerfLaxNor{u}$ comme le \emph{nerf lax normalisé de} $u$. 

L'avantage du nerf lax normalisé sur le nerf lax précédemment défini tient à ce que les calculs le faisant intervenir comportent moins de conditions de cohérence. De plus, comme on s'y attend, ces deux nerfs sont homotopiquement équivalents. Plus précisément, pour toute \deux{}catégorie $\mathdeuxcat{A}$, l'inclusion canonique $\NerfLaxNor{\mathdeuxcat{A}} \hookrightarrow \NerfLax{\mathdeuxcat{A}}$ est une équivalene faible. Comme, pour tout \deux{}foncteur lax normalisé $u : \mathdeuxcat{A} \to \mathdeuxcat{B}$, le diagramme
$$
\xymatrix{
\NerfLaxNor{\mathdeuxcat{A}} 
\ar[r]^{\NerfLaxNor{u}}
\ar[d]
&\NerfLaxNor{\mathdeuxcat{B}} 
\ar[d]
\\
\NerfLax{\mathdeuxcat{A}} 
\ar[r]_{\NerfLax{u}}
&\NerfLax{\mathdeuxcat{B}}
}
$$
est commutatif (les flèches verticales désignant les inclusions canoniques), on en déduit que $\NerfLax{u}$ est une équivalence faible simpliciale si et seulement si c'est le cas de $\NerfLaxNor{u}$. Pour ces affirmations non démontrées dans le présent texte, nous renvoyons à \cite{CCG}. 

Nous définissons maintenant la notion de \deux{}catégorie comma au-dessus (ou au-dessous) d'un simplexe, généralisant celle de \deux{}catégorie au-dessus (ou au-dessous) d'un objet, correspondant au cas des $0$\nobreakdash-simplexes. Nous n'en présentons pas la version la version la plus générale, seul le cas des \DeuxFoncteursStricts{} nous intéressant ici.  Soit donc $u : \mathdeuxcat{A} \to \mathdeuxcat{B}$ un \DeuxFoncteurStrict{} et $b$ un $m$\nobreakdash-simplexe de $\NerfLaxNor{\mathdeuxcat{B}}$. On définit une \deux{}catégorie $\TrancheCoLax{\mathdeuxcat{A}}{u}{b}$ comme suit. Ses objets sont les couples $(a, x)$, avec $a$ un objet de $\mathdeuxcat{A}$ et $x$ un $m+1$\nobreakdash-simplexe de $\NerfLaxNor{\mathdeuxcat{B}}$ tel que $x_{0} = u(a)$ et $d_{0} x = b$. (L'expression $d_{0} x$ désigne la zéroième face de $x$.) Une \un{}cellule de $(a, x)$ vers $(a', x')$ est un couple $(f, y)$ avec $f : a \to a'$ une \un{}cellule de $\mathdeuxcat{A}$ et $y$ un $m+2$\nobreakdash-simplexe de $\mathdeuxcat{B}$ tel que $y_{1,0} = u(f)$, $d_{0} y = x'$ et $d_{1} y = x$. Étant donné les mêmes objets $(a, x)$ et $(a', x')$ ainsi que deux \un{}cellules $(f, y)$ et $(f', y')$ du premier vers le second, une \deux{}cellule de la première vers la seconde est une \deux{}cellule $\alpha : f \Rightarrow f'$ dans $\mathdeuxcat{A}$ telle que, pour tout $0 \leq i \leq m$, on ait l'égalité
$$
y'_{i+2, 1, 0} \left(y_{i+2, 1} \CompDeuxZero u(\alpha)\right) = y_{i+2, 1, 0}
$$
Les diverses unités et compositions sont définies de la façon « évidente ». 

Construisons maintenant un couple de \DeuxFoncteursStricts{} qui sont des équivalences faibles entre $\TrancheCoLax{\mathdeuxcat{A}}{u}{b}$ et $\TrancheCoLax{\mathdeuxcat{A}}{u}{b_{0}}$, suivant toujours en cela \cite{Cegarra}. Notons $R$ le \DeuxFoncteurStrict{} défini par 
$$
\begin{aligned}
\TrancheCoLax{\mathdeuxcat{A}}{u}{b} &\to \TrancheCoLax{\mathdeuxcat{A}}{u}{b_{0}}
\\
(a, x) &\mapsto (a, {x}_{1,0})
\\
(f, y) &\mapsto (f, {y}_{2,1,0})
\\
\gamma &\mapsto \gamma
\end{aligned}
$$
Il admet une section $I : \TrancheCoLax{\mathdeuxcat{A}}{u}{b_{0}} \to \TrancheCoLax{\mathdeuxcat{A}}{u}{b}$ définie comme suit. À tout objet $(a, p)$ de $\TrancheCoLax{\mathdeuxcat{A}}{u}{b_{0}}$, $I$ associe le couple $(a, x)$ défini par $d_{0} x = b$, $x_{0} = u(a)$, $x_{1,0} = p$, $x_{i+1, 0} = b_{i, 0} p$ et $x_{j+1, i+1, 0} = b_{j, i, 0} p$. À toute \un{}cellule $(f, \alpha)$ de $(a, p)$ vers $(a', p')$, $I$ associe le couple $(f, y)$ défini par $y_{i+2, 1, 0} = b_{i, 0} \CompDeuxZero \alpha$. Enfin, pour toute \deux{}cellule $\gamma$ de $\TrancheCoLax{\mathdeuxcat{A}}{u}{b_{0}}$, on pose $I(\gamma) = \gamma$. L'égalité $RI = 1_{\TrancheCoLax{\mathdeuxcat{A}}{u}{b_{0}}}$ est alors évidente. De plus, on construit une \DeuxTransformationStricte{} $\sigma : 1_{\TrancheCoLax{\mathdeuxcat{A}}{u}{b}} \Rightarrow IR$ en posant $\sigma_{(a, x)} = (1_{a}, \tilde{x})$ avec $\tilde{x}_{i+2, 1, 0} = x_{i+1, 1, 0}$. En vertu des lemmes \ref{HomotopieW} et \ref{DeuxTransFoncLax}, $IR$ est donc une équivalence faible. On en déduit que $I$ et $R$ sont des équivalences faibles. On a dualement des définitions et résultats analogues pour les \deux{}catégories $\TrancheLax{\mathdeuxcat{A}}{u}{b}$, $\OpTrancheCoLax{\mathdeuxcat{A}}{u}{b}$ et $\OpTrancheLax{\mathdeuxcat{A}}{u}{b}$.

\begin{corollaire}\label{TranchesSimplexesAspheriques}
Soient $\mathdeuxcat{A}$ une \deux{}catégorie et $a$ un simplexe de $\NerfLaxNor{\mathdeuxcat{A}}$. Les \deux{}catégories $\TrancheCoLax{\mathdeuxcat{A}}{}{a}$, $\TrancheLax{\mathdeuxcat{A}}{}{a}$, $\OpTrancheCoLax{\mathdeuxcat{A}}{}{a}$ et $\OpTrancheLax{\mathdeuxcat{A}}{}{a}$ sont asphériques.
\end{corollaire}

\begin{proof}
C'est une conséquence immédiate des considérations précédentes et du corollaire \ref{TranchesAspheriques}.
\end{proof}

Soient maintenant 
$$
\xymatrix{
\mathdeuxcat{A} 
\ar[rr]^{u}
\ar[dr]_{w}
&&\mathdeuxcat{B}
\ar[dl]^{v}
\\
&\mathdeuxcat{C}
}
$$
un triangle commutatif de \DeuxFoncteursStricts{} et $c$ un simplexe de $\NerfLaxNor{\mathdeuxcat{C}}$. Ces données permettent de définir un \DeuxFoncteurStrict{}
$$
\DeuxFoncTrancheCoLax{u}{c} : \TrancheCoLax{\mathdeuxcat{A}}{w}{c} \to \TrancheCoLax{\mathdeuxcat{B}}{v}{c}
$$
de la façon « évidente ».
On vérifie alors la commutativité du diagramme
$$
\xymatrix{
\TrancheCoLax{\mathdeuxcat{A}}{w}{c}
\ar[r]^{\DeuxFoncTrancheCoLax{u}{c}}
\ar[d]
&\TrancheCoLax{\mathdeuxcat{B}}{v}{c}
\ar[d]
\\
\TrancheCoLax{\mathdeuxcat{A}}{w}{c_{0}}
\ar[r]_{\DeuxFoncTrancheCoLax{u}{c_{0}}}
&\TrancheCoLax{\mathdeuxcat{B}}{v}{c_{0}}
}
$$
dans lequel les flèches verticales désignent les \DeuxFoncteursStricts{} canoniques précédemment notés $R$. On en déduit : 

\begin{lemme}\label{WSurSimplexe}
En conservant les notations ci-dessus, $\DeuxFoncTrancheCoLax{u}{c}$ est une équivalence faible si et seulement si $\DeuxFoncTrancheCoLax{u}{c_{0}}$ en est une. 
\end{lemme}

Sous ces mêmes hypothèses, on définit un ensemble bisimplicial $S_{w}$ par la formule
$$
(S_{w})_{m,n} = \left\{ \left(a \in (\NerfLaxNor{\mathdeuxcat{A}})_{m}, c \in (\NerfLaxNor{\mathdeuxcat{C}}\right)_{m+n+1}), d_{m+1}^{n+1}(c) = (\NerfLaxNor{w}) (a)  \right\}
$$
Autrement dit, on considère les couples $(a, c)$ tels que « l'image de $a$ par $w$ soit le début de $c$ ». On définit de même un ensemble bisimplicial $S_{v}$ par la formule
$$
(S_{v})_{m,n} = \left\{ \left(b \in (\NerfLaxNor{\mathdeuxcat{B}})_{m}, c \in (\NerfLaxNor{\mathdeuxcat{C}})_{m+n+1}\right), d_{m+1}^{n+1}(c) = (\NerfLaxNor{v}) (b)  \right\}
$$
De plus, on considère $\NerfLaxNor{\mathdeuxcat{A}}$ et $\NerfLaxNor{\mathdeuxcat{B}}$ comme des ensembles bisimpliciaux constants sur les colonnes. Autrement dit, pour tout couple $(m, n) \in \mathbb{N}^{2}$, on pose $(\NerfLaxNor{\mathdeuxcat{A}})_{m,n} = (\NerfLaxNor{\mathdeuxcat{A}})_{m}$ et $(\NerfLaxNor{\mathdeuxcat{B}})_{m,n} = (\NerfLaxNor{\mathdeuxcat{B}})_{m}$. On note de même $\NerfLaxNor{u} : \NerfLaxNor{\mathdeuxcat{A}} \to \NerfLaxNor{\mathdeuxcat{B}}$ le morphisme d'ensembles bisimpliciaux induit par $\NerfLaxNor{u}$. On construit alors un diagramme commutatif bisimplicial
$$
\xymatrix{
\NerfLaxNor{\mathdeuxcat{A}}
\ar[r]^{\NerfLaxNor{u}}
&
\NerfLaxNor{\mathdeuxcat{B}}
\\
S_{w}
\ar[u]^{\varphi_{w}}
\ar[r]_{U}
&S_{v}
\ar[u]_{\varphi_{v}}
}
$$
comme suit. Pour tout objet $(a, c) de (S_{w})_{m,n}$, 
$$
(\varphi_{w})_{m,n} (a, c) = a
$$
et 
$$
U_{m,n} (a, c) = \left((\NerfLaxNor{u}) (a), c \right)
$$
La définition de $\varphi_{v}$ est analogue à celle de $\varphi_{w}$.

On note $\Delta$ la catégorie des simplexes. On rappelle que le foncteur diagonal $\Delta \to \Delta \times \Delta$, $[n] \mapsto ([n],[n])$ induit un foncteur $diag : \widehat{\Delta \times \Delta} \to \widehat{\Delta}$. Par construction, on a $diag (\NerfLaxNor{u}) = \NerfLaxNor{u}$. Pour démontrer le résultat souhaité, il suffit donc de vérifier que $diag (U)$, $diag (\varphi_{w})$ et $diag (\varphi_{v})$ sont des équivalences faibles simpliciales. Pour cela, nous utiliserons le lemme classique suivant.

\begin{lemme}\label{LemmeBisimplicial}
Soit $\psi$ un morphisme d'ensembles bisimpliciaux. Supposons que, pour tout entier $m \geq 0$ (\emph{resp.} $n \geq 0$), le morphisme simplicial induit $\psi_{m, \bullet}$ (\emph{resp.} $\psi_{\bullet, n}$) soit une équivalence faible. Alors, $diag (\psi)$ est une équivalence faible.
\end{lemme}

\begin{rem}
Autrement dit, un morphisme bisimplicial qui est une équivalence faible « sur chaque ligne » ou « sur chaque colonne » est une équivalence faible « sur la diagonale ». Pour une démonstration de ce résultat folklorique mais non-trivial, le lecteur pourra consulter \cite[p. 94-95]{QuillenK}, \cite[chapitre XII, paragraphe 4.3]{BK} ou \cite[proposition 1.7]{GoerssJardine}. On pourra également se reporter à \cite[lemme 3.5]{Illusie}, correspondant au cas particulier que nous utiliserons du lemme \ref{LemmeBisimplicial}.
\end{rem} 

En vertu du lemme \ref{LemmeBisimplicial}, pour montrer que $diag (U)$ est une équivalence faible simpliciale, il suffit de montrer que, pour tout entier $n \geq 0$, le morphisme simplicial $U_{\bullet, n} : (S_{w})_{\bullet, n} \to (S_{v})_{\bullet, n}$ en est une. On remarque que la source et le but de ce morphisme s'identifient à
\begin{equation}
\coprod_{c \in (\NerfLaxNor{\mathdeuxcat{C}})_{n}} \NerfLaxNor \left(\TrancheCoLax{\mathdeuxcat{A}}{w}{c}\right)
\end{equation}
et
\begin{equation}
\coprod_{c \in (\NerfLaxNor{\mathdeuxcat{C}})_{n}} \NerfLaxNor (\TrancheCoLax{\mathdeuxcat{B}}{v}{c})
\end{equation}
respectivement, et $U_{\bullet, n}$ s'identifie à 
\begin{equation}
\coprod_{c \in (\NerfLaxNor{\mathdeuxcat{C}})_{n}} \DeuxFoncTrancheCoLax{u}{c}
\end{equation}
En vertu des hypothèses et du lemme \ref{WSurSimplexe}, chaque terme de cette somme est une équivalence faible. Il s'ensuit que $U_{\bullet, n}$ est une équivalence faible. Il en est donc de même de $diag (U)$.

Les raisonnements permettant de montrer que $diag (\varphi_{w})$ et $diag (\varphi_{v})$ sont des équivalences faibles sont analogues. Considérons le cas de $diag (\varphi_{w})$. Pour montrer que c'est une équivalence faible, il suffit, en vertu du lemme \ref{LemmeBisimplicial}, de montrer que, pour tout entier $m \geq 0$, le morphisme simplicial $(\varphi_{w})_{m, \bullet} : (S_{w})_{m, \bullet} \to (\NerfLaxNor{\mathdeuxcat{A}})_{m, \bullet}$ est une équivalence faible. On remarque que la source et le but de ce morphisme s'identifient à 
\begin{equation}
\coprod_{a \in (\NerfLaxNor{\mathdeuxcat{A}})_{m}} \NerfLaxNor (\OpTrancheCoLax{\mathdeuxcat{C}}{}{(\NerfLaxNor{w})(a)})
\end{equation}
et
\begin{equation}
\coprod_{a \in (\NerfLaxNor{\mathdeuxcat{A}})_{m}} \NerfLaxNor \DeuxCatPonct
\end{equation}
respectivement. Le morphisme simplicial $(\varphi_{w})_{m, \bullet}$ s'identifie ainsi à
\begin{equation}
\coprod_{a \in (\NerfLaxNor{\mathdeuxcat{A}})_{m}} \NerfLaxNor (\OpTrancheCoLax{\mathdeuxcat{C}}{}{(\NerfLaxNor{w})(a)} \to \DeuxCatPonct)
\end{equation}
Comme $\OpTrancheCoLax{\mathdeuxcat{C}}{}{(\NerfLaxNor{w})(a)}$ est asphérique en vertu du corollaire \ref{TranchesSimplexesAspheriques}, l'expression ci-dessus est une somme d'équivalences faibles, donc une équivalence faible. Ainsi, $(\varphi_{w})_{m, \bullet}$ est une équivalence faible, donc il en est de même de $diag (\varphi_{w})$. Comme annoncé, il s'ensuit que $\NerfLaxNor{u}$ est une équivalence faible. On a donc démontré le théorème \ref{ThABC}, cas relatif du théorème 2 de \cite{BC}.

\begin{theo}\label{ThABC}
Soit 
$$
\xymatrix{
\mathdeuxcat{A} 
\ar[rr]^{u}
\ar[dr]_{w}
&&\mathdeuxcat{B}
\ar[dl]^{v}
\\
&\mathdeuxcat{C}
}
$$
un triangle commutatif de \DeuxFoncteursStricts{}. Supposons que, pour tout objet $c$ de $\mathdeuxcat{C}$, le \DeuxFoncteurStrict{} $\DeuxFoncTrancheCoLax{u}{c}$ (\emph{resp.} $\DeuxFoncOpTrancheCoLax{u}{c}$, \emph{resp.} $\DeuxFoncTrancheLax{u}{c}$, \emph{resp.} $\DeuxFoncOpTrancheLax{u}{c}$) soit une équivalence faible. Alors $u$ est une é\-qui\-va\-lence faible.
\end{theo}

\begin{corollaire}\label{CorollaireThABC}
Soit $u : \mathdeuxcat{A} \to \mathdeuxcat{B}$ un \DeuxFoncteurStrict{} tel que, pour tout objet $b$ de $\mathdeuxcat{B}$, la \deux{}catégorie $\TrancheCoLax{\mathdeuxcat{A}}{u}{b}$ (\emph{resp.} $\TrancheLax{\mathdeuxcat{A}}{u}{b}$, \emph{resp.} $\OpTrancheCoLax{\mathdeuxcat{A}}{u}{b}$, \emph{resp.} $\OpTrancheLax{\mathdeuxcat{A}}{u}{b}$) soit asphérique. Alors $u$ est une é\-qui\-va\-lence faible.
\end{corollaire}

\begin{proof}
En vertu des hypothèses et du corollaire \ref{TranchesAspheriques}, les flèches orientées vers le bas dans le diagramme commutatif
$$
\xymatrix{
\TrancheCoLax{\mathdeuxcat{A}}{u}{b} 
\ar[rr]^{\DeuxFoncTrancheCoLax{u}{b}}
\ar[dr]_{}
&&\TrancheCoLax{\mathdeuxcat{B}}{}{b} 
\ar[dl]^{}
\\
&\DeuxCatPonct
}
$$
sont des équivalences faibles. La conclusion découle d'un argument de « 2 sur 3 » et du théorème \ref{ThABC}.
\end{proof}

On rappelle que, bien que cela ne soit pas la définition la plus répandue, un foncteur $u : A \to B$ est un adjoint à gauche (autrement dit, il possède un adjoint à droite) si et seulement si, pour tout objet $b$ de $B$, la catégorie $A/b$ admet un objet final. Les notions définies plus haut rendent alors la définition \ref{DefPreadjoints} naturelle.

\begin{df}\label{DefPreadjoints}
On dira qu'un \DeuxFoncteurStrict{} $u : \mathdeuxcat{A} \to \mathdeuxcat{B}$ est un \deux{}\emph{préadjoint à gauche colax} si, pour tout objet $b$ de $\mathdeuxcat{B}$, la \deux{}catégorie $\TrancheCoLax{\mathdeuxcat{A}}{u}{b}$ admet un objet admettant un objet final. 

On dira qu'un \DeuxFoncteurStrict{} $u : \mathdeuxcat{A} \to \mathdeuxcat{B}$ est un \deux{}\emph{précoadjoint à gauche lax} si $\DeuxFoncDeuxOp{u}$ est un préadjoint à gauche colax. Cette condition équivaut à la suivante : pour tout objet $b$ de $\mathdeuxcat{B}$, la \deux{}catégorie $\TrancheLax{\mathdeuxcat{A}}{u}{b}$ admet un objet admettant un objet initial. 

On dira qu'un \DeuxFoncteurStrict{} $u : \mathdeuxcat{A} \to \mathdeuxcat{B}$ est un \emph{préadjoint à droite colax} si $\DeuxFoncUnOp{u}$ est un préadjoint à gauche colax. Cette condition équivaut à la suivante : pour tout objet $b$ de $\mathdeuxcat{B}$, la \deux{}catégorie $\DeuxCatUnOp{(\OpTrancheCoLax{\mathdeuxcat{A}}{u}{b})}$ admet un objet admettant un objet final. 

On dira qu'un \DeuxFoncteurStrict{} $u : \mathdeuxcat{A} \to \mathdeuxcat{B}$ est un \emph{précoadjoint à droite lax} si $\DeuxFoncToutOp{u}$ est un préadjoint à gauche colax. Cette condition équivaut à la suivante : pour tout objet $b$ de $\mathdeuxcat{B}$, la \deux{}catégorie $\DeuxCatUnOp{(\OpTrancheLax{\mathdeuxcat{A}}{u}{b})}$ admet un objet admettant un objet initial.
\end{df}

\begin{rem}
On trouve, à la page 941 de \cite{BettiPower}, l'exacte définition de ce que nous appelons « préadjoint à droite lax », sous le nom de \DeuxFoncteurLax{}  \emph{induisant un adjoint à gauche local}. Comme, suivant les auteurs eux-mêmes, la définition que nous donnons équivaut à \cite[définition 4.1]{BettiPower}, cette dernière définition fournit une caractérisation des morphismes vérifiant la propriété universelle que nous privilégions. Toujours à la lecture de \cite{BettiPower}, nous avons appris que cette propriété universelle avait été dégagée par Bunge comme une généralisation de la notion d'extension de Kan \cite[p. 357]{Bunge}. Soulignons toutefois que \cite[définition 4.1]{BettiPower} n'est pas la définition de ce que les auteurs de \cite{BettiPower} appellent \emph{adjoint local}, notion dont la définition est \cite[définition 3.1]{BettiPower} et qui présente l'avantage d'être stable par composition. 
\end{rem} 

\begin{lemme}\label{PreadjointEquiFaible}
Tout préadjoint à gauche lax (\emph{resp.} préadjoint à gauche colax, \emph{resp.} préadjoint à droite lax, \emph{resp.} préadjoint à droite colax) est une équivalence faible. 
\end{lemme}

\begin{proof}
Cela résulte des corollaires \ref{TranchesAspheriques} et \ref{CorollaireThABC}.
\end{proof}

\begin{df}\label{DefFibre}
Soient $u : \mathdeuxcat{A} \to \mathdeuxcat{B}$ un \DeuxFoncteurStrict{} et $b$ un objet de $\mathdeuxcat{B}$. On appelle \emph{fibre de} $u$ \emph{au-dessus de} $b$ la \deux{}catégorie, que l'on notera $\Fibre{\mathdeuxcat{A}}{u}{b}$, dont les objets sont les objets $a$ de $\mathdeuxcat{A}$ tels que $u(a) = b$, dont les \un{}cellules de $a$ vers $a'$ sont les \un{}cellules $f $ de $a$ vers $a'$ dans $\mathdeuxcat{A}$ telles que $u(f) = 1_{b}$, et dont les \deux{}cellules de $f$ vers $f'$ sont les \deux{}cellules $\alpha$ de $f$ vers $f'$ telles que $u(\alpha) = 1_{1_{b}}$, les diverses compositions et unités étant héritées de celles de $\mathdeuxcat{A}$ de façon évidente.
\end{df} 

Le lemme \ref{FibresCoOp} se vérifie immédiatement.

\begin{lemme}\label{FibresCoOp}
Pour tout \DeuxFoncteurStrict{} $u : \mathdeuxcat{A} \to \mathdeuxcat{B}$ et tout objet $b$ de $\mathdeuxcat{B}$, la \deux{}catégorie $\Fibre{(\DeuxCatUnOp{\mathdeuxcat{A}})}{(\DeuxFoncUnOp{u})}{b}$ (\emph{resp.} $\Fibre{(\DeuxCatDeuxOp{\mathdeuxcat{A}})}{(\DeuxFoncDeuxOp{u})}{b}$, \emph{resp.} $\Fibre{(\DeuxCatToutOp{\mathdeuxcat{A}})}{(\DeuxFoncToutOp{u})}{b}$) s'identifie canoniquement à $\DeuxCatUnOp{(\Fibre{\mathdeuxcat{A}}{u}{b})}$ (\emph{resp.} $\DeuxCatDeuxOp{(\Fibre{\mathdeuxcat{A}}{u}{b})}$, \emph{resp.} $\DeuxCatToutOp{(\Fibre{\mathdeuxcat{A}}{u}{b})}$).
\end{lemme}

On rappelle qu'un foncteur $u : A \to B$ est une \emph{préfibration} si et seulement si, pour tout objet $b$ de $\mathdeuxcat{B}$, le foncteur canonique $A_{b} \to b \backslash A$ est un adjoint à gauche (autrement dit, il admet un adjoint à droite). Nous sommes donc naturellement amenés, par analogie, à poser la définition \ref{DefPrefibration}. 

\begin{df}\label{DefPrefibration}
On dira qu'un \DeuxFoncteurStrict{} $u : \mathdeuxcat{A} \to \mathdeuxcat{B}$ est une \deux{}\emph{préfibration} si, pour tout objet $b$ de $\mathdeuxcat{B}$, le \DeuxFoncteurStrict{} canonique 
$$
\begin{aligned}
J_{b} : \Fibre{\mathdeuxcat{A}}{u}{b} &\to \OpTrancheCoLax{\mathdeuxcat{A}}{u}{b}
\\
a &\mapsto (a, 1_{b})
\\
f &\mapsto (f, 1_{1_{b}})
\\
\alpha &\mapsto \alpha
\end{aligned}
$$
est un préadjoint à gauche lax.

On dira qu'un  \DeuxFoncteurStrict{} $u : \mathdeuxcat{A} \to \mathdeuxcat{B}$ est une \deux{}\emph{préopfibration} si $\DeuxFoncUnOp{u}$ est une préfibration, autrement dit si, pour tout objet $b$ de $\mathdeuxcat{B}$, le morphisme canonique $\Fibre{\mathdeuxcat{A}}{u}{b} \to \TrancheCoLax{\mathdeuxcat{A}}{u}{b}$ est un préadjoint à droite lax.

On dira qu'un  \DeuxFoncteurStrict{} $u : \mathdeuxcat{A} \to \mathdeuxcat{B}$ est une \deux{}\emph{précofibration} si $\DeuxFoncDeuxOp{u}$ est une préfibration, autrement dit si, pour tout objet $b$ de $\mathdeuxcat{B}$, le morphisme canonique $\Fibre{\mathdeuxcat{A}}{u}{b} \to \OpTrancheLax{\mathdeuxcat{A}}{u}{b}$ est un préadjoint à gauche colax. 

On dira qu'un  \DeuxFoncteurStrict{} $u : \mathdeuxcat{A} \to \mathdeuxcat{B}$ est une \deux{}\emph{précoopfibration} si $\DeuxFoncToutOp{u}$ est une préfibration, autrement dit si, pour tout objet $b$ de $\mathdeuxcat{B}$, le morphisme canonique $\Fibre{\mathdeuxcat{A}}{u}{b} \to \TrancheLax{\mathdeuxcat{A}}{u}{b}$ est un préadjoint à droite colax.   
\end{df}

\begin{exemple}\label{ProjectionPrefibration}
Des calculs quelque peu pénibles mais sans difficulté permettent de vérifier que toute projection est une préfibration (et donc également une préopfibration, une précofibration et une précoopfibration).
\end{exemple}

\begin{lemme}\label{PrefibrationFibresAsph}
Toute préfibration (\emph{resp.} préopfibration, \emph{resp.} précofibration, \emph{resp.} précoopfibration) à fibres asphériques est une équivalence faible.
\end{lemme}

\begin{proof}
Cela résulte du lemme \ref{PreadjointEquiFaible} et du corollaire \ref{CorollaireThABC}. 
\end{proof}

\section{Intégration des 2-foncteurs}\label{SectionIntegration}

\begin{paragr}
À tout \DeuxFoncteurStrict{} $F : \mathdeuxcat{A} \to \DeuxCatDeuxCat$, on associe une \deux{}catégorie $\DeuxInt{F}$ comme suit. 

Les objets de $\DeuxInt{F}$ sont les couples $(a, x)$, avec $a$ objet de $\mathdeuxcat{A}$ et $x$ objet de la \deux{}catégorie $F(a)$. 

Les \un{}cellules de $(a, x)$ vers $(a', x')$ dans $\DeuxInt{F}$ sont les couples
$$
\left(f : a \to a', r : F(f)(x) \to x'\right)
$$
dans lesquels $f$ est une \un{}cellule de $\mathdeuxcat{A}$ et $r$ une \un{}cellule de $F(a')$. 

Les \deux{}cellules de $(f : a \to a', r : F(f)(x) \to x')$ vers $(g : a \to a', s : F(g)(x) \to x')$ dans $\DeuxInt{F}$ sont les couples
$$
(\gamma : f \Rightarrow g, \varphi : r \Rightarrow s (F(\gamma))_{x})
$$
dans lesquels $\gamma$ est une \deux{}cellule dans $\mathdeuxcat{A}$ et $\varphi$ une \deux{}cellule dans $F(a')$.

L'identité d'un objet est donnée par la formule
$$
1_{(a, x)} = (1_{a}, 1_{x})
$$

La composée des \un{}cellules est donnée par la formule
$$
(f', r') (f, r) = (f'f, r' F(f')(r))
$$ 

L'identité d'une \un{}cellule est donnée par la formule
$$
1_{(f, r)} = (1_{f}, 1_{r})
$$

La composée verticale des \deux{}cellules est donnée par la formule
$$
(\delta, \psi) \CompDeuxUn (\gamma, \varphi) = (\delta \CompDeuxUn \gamma, (\psi \CompDeuxZero (F(\gamma))_{x}) \CompDeuxUn \varphi)
$$

La composée horizontale des \deux{}cellules est donnée par la formule
$$
(\gamma', \varphi') \CompDeuxZero (\gamma, \varphi) = (\gamma' \CompDeuxZero \gamma, \varphi' \CompDeuxZero F(f') (\varphi))
$$

On vérifie sans difficulté que cela définit bien une \deux{}catégorie.
\end{paragr}

Sous les mêmes hypothèses, l'application évidente
$$
\begin{aligned}
\DeuxInt{F} &\to \mathdeuxcat{A}
\\
(a,x) &\mapsto a
\\
(f,r) &\mapsto f
\\
(\gamma, \varphi) &\mapsto \gamma
\end{aligned}
$$ 
définit un \DeuxFoncteurStrict{} $P_{F} : \DeuxInt{F} \to \mathdeuxcat{A}$.

\begin{prop}\label{ProjIntCoOpFib}
Pour toute \deux{}catégorie $\mathdeuxcat{A}$ et tout \DeuxFoncteurStrict{}  $F : \mathdeuxcat{A} \to \DeuxCatDeuxCat$, la projection associée à l'intégrale de $F$
$$
P_{F} : \DeuxInt{F} \to \mathdeuxcat{A}
$$ 
est une précoopfibration.
\end{prop}

\begin{proof}
Soit $a$ un objet de $\mathdeuxcat{A}$. Décrivons les objets, \un{}cellules et \deux{}cellules de la \deux{}catégorie $\TrancheLax{\DeuxInt{F}}{P_{F}}{a}$. (On laisse au lecteur le soin d'expliciter les détails si besoin est.)

Les objets en sont les $((a', x'), p : a' \to a)$.

Les \un{}cellules de $((a', x'), p : a' \to a)$ vers $((a'', x''), p' : a'' \to a)$ sont les $((f : a' \to a'', r : F(f)(x') \to x''), \sigma : p \Rightarrow p'f)$.

Les \deux{}cellules de $((f : a' \to a'', r : F(f)(x') \to x''), \sigma : p \Rightarrow p'f)$ vers $((g : a' \to a'', s : F(g)(x') \to x''), \sigma' : p \Rightarrow p'g)$ sont les $(\gamma : f \Rightarrow g, \varphi : r \Rightarrow s (F(\gamma))_{x'})$ tels que $(p' \CompDeuxZero \gamma) \CompDeuxUn \sigma = \sigma'$. 

Considérons le \DeuxFoncteurStrict{}
$$
\begin{aligned}
J_{a} : \Fibre{\left(\DeuxInt{F}\right)}{P_{F}}{a} &\to \TrancheLax{\left(\DeuxInt{F}\right)}{P_{F}}{a}
\\
(a, x) &\mapsto ((a, x), 1_{a})
\\
(1_{a}, r) &\mapsto ((1_{a}, r), 1_{1_{a}})
\\
(1_{1_{a}}, \varphi) &\mapsto (1_{1_{a}}, \varphi)
\end{aligned}
$$ 

Soit $((a', x'), p : a' \to a)$ un objet de $\TrancheLax{(\DeuxInt{F})}{P_{F}}{a}$. Ainsi, $a$, $a'$, $x'$ et $p$ sont désormais fixés. Décrivons les objets, \un{}cellules et \deux{}cellules de la \deux{}catégorie $\OpTrancheCoLax{\Fibre{(\DeuxInt{F})}{P_{F}}{a}}{J_{a}}{((a', x'), p)}$. (Comme ci-dessus, on laisse au lecteur le soin d'expliciter les détails si besoin est.)

Les objets en sont les $((a, x), ((q : a' \to a, r : F(q)(x') \to x), \sigma : p \Rightarrow q))$.

Les \un{}cellules de $((a, x), ((q : a' \to a, r : F(q)(x') \to x), \sigma : p \Rightarrow q))$ vers $((a, x''), ((q' : a' \to a, r' : F(q')(x') \to x''), \sigma' : p \Rightarrow q'))$ sont les $(s : x \to x'', (\gamma : q \Rightarrow q', \varphi : sr \Rightarrow r' (F(\gamma))_{x'}))$ tels que $\gamma \CompDeuxUn \sigma = \sigma'$. 

Les \deux{}cellules de $(s : x \to x'', (\gamma : q \Rightarrow q', \varphi : sr \Rightarrow r' (F(\gamma))_{x'}))$ (vérifiant $\gamma \CompDeuxUn \sigma = \sigma'$) vers $(t : x \to x'', (\mu : q \Rightarrow q', \psi : tr \Rightarrow r' (F(\mu))_{x'}))$ (vérifiant $\mu \CompDeuxUn \sigma = \sigma'$) correspondent aux \deux{}cellules $\tau : s \Rightarrow t$ telles que $\mu = \gamma$ et $\psi \CompDeuxUn (\tau \CompDeuxZero r) = \varphi$. (L'égalité $\gamma = \mu$ est donc une condition nécessaire à l'existence d'une telle \deux{}cellule.) 

Dans la \deux{}catégorie $\OpTrancheCoLax{\Fibre{(\DeuxInt{F})}{P_{F}}{a}}{J_{a}}{((a', x'), p)}$, on distingue l'objet 
$$
((a, F(p)(x')) ((p, 1_{F(p)(x')}), 1_{p}))
$$ 
Soit $((a, x''), ((q' : a' \to a, r' : F(q')(x') \to x''), \sigma' : p \Rightarrow q'))$ un objet quelconque de la \deux{}catégorie $\OpTrancheCoLax{\Fibre{(\DeuxInt{F})}{P_{F}}{a}}{J_{a}}{((a', x'), p)}$. On distingue la \un{}cellule $(r' (F(\sigma'))_{x'}, (\sigma', 1_{r' (F(\sigma'))_{x'}}))$ de $((a, F(p)(x')), ((p, 1_{F(p)(x')}), 1_{p}))$ vers $((a, x''), ((q', r'), \sigma'))$. En effet, la condition à vérifier n'est autre que $\sigma' \CompDeuxUn 1_{p} = \sigma'$, qui est trivialement vérifiée. 

Soit $(s : F(p)(x') \to x'', (\gamma : p \Rightarrow q', \varphi : s \Rightarrow r' (F(\sigma'))_{x'}))$ une \un{}cellule quelconque de $((a, F(p)(x')), ((p, 1_{F(p)(x')}), 1_{p}))$ vers $((a, x''), ((q', r'), \sigma'))$. On a donc en particulier $\gamma = \sigma'$. Les \deux{}cellules de $(s, (\gamma, \varphi))$ vers $(r' (F(\sigma'))_{x'}, (\sigma', 1_{r' (F(\sigma'))_{x'}}))$ sont les \deux{}cellules $\tau : s \Rightarrow r' (F(\sigma'))_{x'}$ telles que $\gamma = \sigma'$ (égalité déjà vérifiée par hypothèse) et $1_{r' (F(\sigma'))_{x'}} (\tau \CompDeuxZero 1_{F(p)(x')}) = \varphi$, c'est-à-dire $\tau = \varphi$. Il est clair que ces conditions impliquent l'existence et l'unicité d'une telle \deux{}cellule, à savoir $\varphi$. 

Ainsi, la \deux{}catégorie $\DeuxCatUnOp{\OpTrancheCoLax{\Fibre{(\DeuxInt{F})}{P_{F}}{a}}{J_{a}}{((a', x'), p)}}$ admet un objet admettant un objet final. Par définition, le résultat suit.
\end{proof}

\begin{paragr}
Étant donné deux \DeuxFoncteursStricts{} $u$ et $v$ d'une \deux{}catégorie $\mathdeuxcat{A}$ vers $\DeuxCatDeuxCat$ et une \DeuxTransformationStricte{} $\sigma$ de $u$ vers $v$, on obtient un \DeuxFoncteurStrict{} $\DeuxInt{\sigma} : \DeuxInt{u} \to \DeuxInt{v}$ en posant :
$$
\begin{aligned}
\DeuxInt{\sigma} : \DeuxInt{u} &\to \DeuxInt{v}
\\
(a, x) &\mapsto (a, \sigma_{a}(x))
\\
(f, r) &\mapsto (f, \sigma_{a'}(r))
\\
(\gamma, \varphi) &\mapsto (\gamma, \sigma_{a'} (\varphi))
\end{aligned}
$$
\end{paragr}

\begin{rem}
Soient $\mathdeuxcat{A}$ une \deux{}catégorie, $u$ et $v$ des \DeuxFoncteursStricts{} de $\mathdeuxcat{A}$ vers $\DeuxCatDeuxCat$, et $\sigma$ une \DeuxTransformationStricte{} de $u$ vers $v$. Alors, le diagramme de \DeuxFoncteursStricts{}
$$
\xymatrix{
\DeuxInt{u}
\ar[rr]^{\DeuxInt{\sigma}}
\ar[dr]_{P_{u}}
&&\DeuxInt{v}
\ar[dl]^{P_{v}}
\\
&\mathdeuxcat{A}
}
$$
dans lequel les flèches diagonales orientées vers le bas désignent les projections canoniques, est commutatif. On a donc notamment, pour tout objet $a$ de $\mathdeuxcat{A}$, un \DeuxFoncteurStrict{} 
$$
\DeuxFoncTrancheLax{\left(\DeuxInt{\sigma}\right)}{a} : \TrancheLax{\left(\DeuxInt{u}\right)}{P_{u}}{a} \to \TrancheLax{\left(\DeuxInt{v}\right)}{P_{v}}{a}
$$
\end{rem}

Les lemmes \ref{KikiDeMontparnasse} et \ref{PierreDac} sont immédiats. 

\begin{lemme}\label{KikiDeMontparnasse}
Soient $\mathdeuxcat{A}$ une \deux{}catégorie, $u$ et $v$ des \DeuxFoncteursStricts{} de $\mathdeuxcat{A}$ vers $\DeuxCatDeuxCat$, $\sigma$ une \DeuxTransformationStricte{} de $u$ vers $v$ et $a$ un objet de $\mathdeuxcat{A}$. Il existe alors un diagramme commutatif de \DeuxFoncteursStricts{}
$$
\xymatrix{
u(a)
\ar[r]^{\sigma_{a}}
\ar[d]
&v(a)
\ar[d]
\\
\Fibre{(\DeuxInt{u})}{P_{u}}{a}
\ar[r]_{(\DeuxInt{\sigma})_{a}}
&\Fibre{(\DeuxInt{v})}{P_{v}}{a}
}
$$
dont les flèches verticales sont des isomorphismes.
\end{lemme}

\begin{lemme}\label{PierreDac}
Soient $\mathdeuxcat{A}$ une \deux{}catégorie, $u$ et $v$ des \DeuxFoncteursStricts{} de $\mathdeuxcat{A}$ vers $\DeuxCatDeuxCat$ et $\sigma$ une \DeuxTransformationStricte{} de $u$ vers $v$. Alors, pour tout objet $a$ de $\mathdeuxcat{A}$, le diagramme\footnote{Dans lequel on a commis l'abus de noter de la même façon les deux flèches verticales et dans lequel $(\DeuxInt{\sigma})_{a}$ désigne le \DeuxFoncteurStrict{} induit entre les fibres.} de \DeuxFoncteursStricts{}
$$
\xymatrix{
\Fibre{(\DeuxInt{u})}{P_{u}}{a}
\ar[rr]^{(\DeuxInt{\sigma})_{a}}
\ar[d]_{J_{a}}
&&\Fibre{(\DeuxInt{v})}{P_{v}}{a}
\ar[d]^{J_{a}}
\\
\TrancheLax{(\DeuxInt{u})}{P_{u}}{a}
\ar[rr]_{\DeuxFoncTrancheLax{(\DeuxInt{\sigma})}{a}}
&&\TrancheLax{(\DeuxInt{v})}{P_{v}}{a}
}
$$
est commutatif. 
\end{lemme}

\begin{lemme}\label{LeoCampion}
Soient $\mathdeuxcat{A}$ une \deux{}catégorie, $u$ et $v$ des \DeuxFoncteursStricts{} de $\mathdeuxcat{A}$ vers $\DeuxCatDeuxCat$, et $\sigma$ une \DeuxTransformationStricte{} de $u$ vers $v$. Supposons que, pour tout objet $a$ de $\mathdeuxcat{A}$, le \DeuxFoncteurStrict{} $\sigma_{a} : u(a) \to v(a)$ soit une équivalence faible. Alors, le \DeuxFoncteurStrict{} $\DeuxInt{\sigma} : \DeuxInt{u} \to \DeuxInt{v}$ est une équivalence faible.
\end{lemme}

\begin{proof}
Comme $P_{u}$ et $P_{v}$ sont des précoopfibrations (proposition \ref{ProjIntCoOpFib}), les flèches verticales du diagramme ci-dessus sont des préadjoints à droite colax (par définition), donc des équivalences faibles (lemme \ref{PreadjointEquiFaible}). En vertu des lemmes \ref{KikiDeMontparnasse} et \ref{PierreDac} et de la saturation faible de $\DeuxLocFond{W}$, les hypothèses impliquent donc que le \DeuxFoncteurStrict{} $\DeuxFoncTrancheLax{(\DeuxInt{\sigma})}{a}$ est une équivalence faible pour tout objet $a$ de $\mathdeuxcat{A}$, d'où le résultat en vertu du théorème \ref{ThABC}.  
\end{proof}

\begin{paragr}
Soit $u : \mathdeuxcat{A} \to \mathdeuxcat{B}$ un \DeuxFoncteurLax{}. On lui associe un \DeuxFoncteurStrict{} $\DeuxFoncteurTranche{u} : \mathdeuxcat{B} \to \DeuxCatDeuxCat$ comme suit.  

Pour tout objet $b$ de $\mathdeuxcat{B}$, on pose $\DeuxFoncteurTranche{u}(b) = \TrancheLax{\mathdeuxcat{A}}{u}{b}$. 

Soit $f : b \to b'$ une \un{}cellule de $\mathdeuxcat{B}$. On définit le \DeuxFoncteurStrict{} $\DeuxFoncteurTranche{u}(f)$ par
$$
\begin{aligned}
\DeuxFoncteurTranche{u}(f) : \TrancheLax{\mathdeuxcat{A}}{u}{b} &\to \TrancheLax{\mathdeuxcat{A}}{u}{b'}
\\
(a, p : u(a) \to b) &\mapsto (a, fp : u(a) \to b \to b')
\\
(g : a \to a', \alpha : p \Rightarrow p' u(g)) &\mapsto (g : a \to a', f \CompDeuxZero \alpha : fp \Rightarrow f p' u(g))
\\
\beta &\mapsto \beta
\end{aligned}
$$ 

Soient $f$ et $f'$ deux \un{}cellules de $b$ vers $b'$ et $\gamma : f \Rightarrow f'$ une \deux{}cellule dans $\mathdeuxcat{B}$. On définit une \DeuxTransformationStricte{} $\DeuxFoncteurTranche{u}(\gamma) : \DeuxFoncteurTranche{u}(f) \Rightarrow \DeuxFoncteurTranche{u}(f')$ en posant
$$
(\DeuxFoncteurTranche{u}(\gamma))_{(a,p)} = (1_{a}, \gamma \CompDeuxZero p \CompDeuxZero u_{a})
$$

Les diverses conditions de cohérence se vérifient sans difficulté.
\end{paragr}

\begin{paragr}
Soit 
$$
\xymatrix{
\mathdeuxcat{A}  
\ar[rr]^{u} 
\ar[dr]_{w} 
&&\mathdeuxcat{B}
\dtwocell<\omit>{<7.3>\sigma} 
\ar[dl]^{v}
\\ 
& 
\mathdeuxcat{C}
&{}
}
$$
un diagramme dans lequel $u$, $v$ et $w$ sont des \DeuxFoncteursStricts{} et $\sigma$ est une \DeuxTransformationOpLax{}. Ces données nous fournissent notamment, par le procédé décrit ci-dessus, des \DeuxFoncteursStricts{} $\DeuxFoncteurTranche{w}$ et $\DeuxFoncteurTranche{v}$ de $\mathdeuxcat{C}$ vers $\DeuxCatDeuxCat$. On définit une \DeuxTransformationStricte{} $\DeuxFoncteurTranche{\sigma} :\DeuxFoncteurTranche{w} \Rightarrow \DeuxFoncteurTranche{v}$ en posant 
$$
(\DeuxFoncteurTranche{\sigma})_{c} = \DeuxFoncTrancheLaxCoq{u}{\sigma}{c}
$$
pour tout objet $c$ de $\mathdeuxcat{C}$. Ici encore, la vérification des conditions de cohérence, bien que fastidieuse, ne présente guère de difficulté. \footnote{Il s'agit toujours d'un cas particulier d'une construction relative aux « \deux{}catégories commas » ; voir \cite[section 1.9.5]{TheseMoi}.}
\end{paragr}

\begin{paragr}
Soit $w : \mathdeuxcat{A} \to \mathdeuxcat{C}$ un  \DeuxFoncteurStrict{}. En vertu des résultats généraux déjà dégagés sur l'intégration, la projection canonique\footnote{Que nous n'appelons pas $P_{\DeuxFoncteurTranche{w}}$ par commodité.} $P_{\mathdeuxcat{C}} : \DeuxInt{\DeuxFoncteurTranche{w}} \to \mathdeuxcat{C}$ est une précoopfibration. On se propose d'expliciter la structure de $\DeuxInt{\DeuxFoncteurTranche{w}}$ et de vérifier qu'il existe également une projection canonique sur $\mathdeuxcat{A}$, que l'on notera $Q_{\mathdeuxcat{A}} : \DeuxInt{\DeuxFoncteurTranche{w}} \to \mathdeuxcat{A}$, qui est une préfibration. Autrement dit, $\DeuxInt{\DeuxFoncteurTranche{w}}$ est non seulement une \deux{}catégorie précoopfibrée sur $\mathdeuxcat{C}$, mais c'est également une \deux{}catégorie préfibrée sur $\mathdeuxcat{A}$. Au cours des descriptions suivantes, on omet quelques détails que le lecteur rétablira de lui-même s'il le souhaite. 

Les objets de $\DeuxInt{\DeuxFoncteurTranche{w}}$ sont de la forme $(c, (a, p : w(a) \to c))$.

Les \un{}cellules de $(c, (a, p))$ vers $(c', (a', p'))$ sont de la forme $(k : c \to c', (f : a \to a', \gamma : kp \Rightarrow p' w(f)))$.

Les \deux{}cellules de $(k : c \to c', (f : a \to a', \gamma : kp \Rightarrow p' w(f)))$ vers $(l : c \to c', (g : a \to a', \delta : lp \Rightarrow p' w(g)))$ sont de la forme $(\epsilon : k \Rightarrow l, \varphi : f \Rightarrow g)$ et tels que $(p' \CompDeuxZero w(\varphi)) \CompDeuxUn \gamma = \delta \CompDeuxUn (\epsilon \CompDeuxZero p)$.

L'identité de l'objet $(c, (a, p))$ est donnée par $(1_{c}, (1_{a}, 1_{p}))$. 

La composée de \un{}cellules 
$
(k', (f', \gamma')) (k, (f, \gamma))
$, quand elle fait sens, est donnée par la formule 
$$(k', (f', \gamma')) (k, (f, \gamma)) = (k'k, f'f, (\gamma' \CompDeuxZero w(f)) (k' \CompDeuxZero \gamma))$$

L'identité de la \un{}cellule $(k, (f, \gamma))$ est donnée par $1_{(k, (f, \gamma))} = (1_{k}, 1_{f})$.

Étant donné deux \deux{}cellules $(\lambda, \psi)$ et $(\epsilon, \varphi)$ telles que la composée verticale $(\lambda, \psi) (\epsilon, \varphi)$ fasse sens, cette dernière est donnée par la formule
$$
(\lambda, \psi) (\epsilon, \varphi) = (\lambda \epsilon, \psi \varphi)
$$

Étant donné deux \deux{}cellules $(\epsilon, \varphi)$ et $(\epsilon', \varphi')$ telles que la composée horizontale $(\epsilon', \varphi') \CompDeuxZero (\epsilon, \varphi)$ fasse sens, cette dernière est donnée par la formule
$$
(\epsilon', \varphi') \CompDeuxZero (\epsilon, \varphi) = (\epsilon' \CompDeuxZero \epsilon, \varphi' \CompDeuxZero \varphi)
$$

Cela termine la description de la \deux{}catégorie $\DeuxInt{\DeuxFoncteurTranche{w}}$. On définit un \DeuxFoncteurStrict{} $Q_{\mathdeuxcat{A}} : \DeuxInt{\DeuxFoncteurTranche{w}} \to \mathdeuxcat{A}$ par
$$
\begin{aligned}
Q_{\mathdeuxcat{A}} : \DeuxInt{\DeuxFoncteurTranche{w}} &\to \mathdeuxcat{A}
\\
(c, (a, p)) &\mapsto a
\\
(k, (f, \gamma)) &\mapsto f
\\
(\epsilon, \varphi) &\mapsto \varphi
\end{aligned}
$$
\end{paragr}

\begin{lemme}\label{QPrefibration}
Le \DeuxFoncteurStrict{} $Q_{\mathdeuxcat{A}} : \DeuxInt{\DeuxFoncteurTranche{w}} \to \mathdeuxcat{A}$ défini ci-dessus est une préfibration. 
\end{lemme}

\begin{proof}
Le résultat se vérifie directement par des calculs analogues à ceux figurant dans la dé\-monstra\-tion de la proposition \ref{ProjIntCoOpFib} (voir aussi la remarque \ref{RemDual}). Les détails sont laissés au lecteur. 
\end{proof}

\begin{rem}\label{RemDual}
Dans \cite{TheseMoi}, on montre le lemme \ref{QPrefibration} par un argument de dualité. Cela nécessite toutefois le développement quelque peu fastidieux de notions duales à celle de l'intégration, ce que nous avons estimé préférable d'épargner au lecteur.
\end{rem}

\begin{lemme}\label{FibreQ}
Pour tout objet $a$ de $\mathdeuxcat{A}$, la fibre $\Fibre{(\DeuxInt{\DeuxFoncteurTranche{w}})}{Q_{\mathdeuxcat{A}}}{a}$ est canoniquement isomorphe à $\OpTrancheCoLax{\mathdeuxcat{C}}{}{w(a)}$.
\end{lemme}

\begin{proof}
C'est immédiat. 
\end{proof}

\begin{corollaire}\label{QEquiFaible}
Le \DeuxFoncteurStrict{} $Q_{\mathdeuxcat{A}}$ est une équivalence faible.
\end{corollaire}

\begin{proof}
En vertu du lemme \ref{FibreQ}, de l'exemple \ref{ExemplesOF} et du lemme \ref{OFAspherique}, les fibres de $Q_{\mathdeuxcat{A}}$ sont asphériques. Le \DeuxFoncteurStrict{} $Q_{\mathdeuxcat{A}}$ est donc une préfibration (proposition \ref{QPrefibration}) à fibres asphériques, donc une équivalence faible (lemme \ref{PrefibrationFibresAsph}). 
\end{proof}

\section{Le cas strict}\label{SectionCasStrict}

\begin{theo}\label{ThAStrictCoq}
Soit 
$$
\xymatrix{
\mathdeuxcat{A}  
\ar[rr]^{u} 
\ar[dr]_{w} 
&&\mathdeuxcat{B}
\dtwocell<\omit>{<7.3>\sigma} 
\ar[dl]^{v}
\\ 
& 
\mathdeuxcat{C}
&{}
}
$$
un diagramme dans lequel $u$, $v$ et $w$ sont des \DeuxFoncteursStricts{} et $\sigma$ est une \DeuxTransformationStricte{}. Si, pour tout objet $c$ de $\mathdeuxcat{C}$, le \DeuxFoncteurStrict{} $\DeuxFoncTrancheLaxCoq{u}{\sigma}{c} : \TrancheLax{\mathdeuxcat{A}}{w}{c} \to \TrancheLax{\mathdeuxcat{B}}{v}{c}$ est une équivalence faible, alors $u$ est une équivalence faible.
\end{theo}

\begin{proof}
On vérifie sans difficulté la commutativité du diagramme de \DeuxFoncteursStricts{}
$$
\xymatrix{
\DeuxInt{\DeuxFoncteurTranche{w}}
\ar[rr]^{\DeuxInt{\DeuxFoncteurTranche{\sigma}}}
\ar[d]_{Q_{\mathdeuxcat{A}}}
&&\DeuxInt{\DeuxFoncteurTranche{v}}
\ar[d]^{Q_{\mathdeuxcat{B}}}
\\
\mathdeuxcat{A}
\ar[rr]_{u}
&&\mathdeuxcat{B}
}
$$
En vertu des hypothèses et du lemme \ref{LeoCampion}, le \DeuxFoncteurStrict{} $\DeuxInt{\DeuxFoncteurTranche{\sigma}}$ est une équivalence faible. Comme $Q_{\mathdeuxcat{A}}$ et $Q_{\mathdeuxcat{B}}$ sont des équivalences faibles en vertu du lemme \ref{QEquiFaible}, la saturation faible de $\DeuxLocFond{W}$ permet de conclure.
\end{proof}

\section{L'adjonction de Bénabou}\label{SectionAdjonctionBenabou}

Cette section a pour objet l'étude des premières propriétés homotopiques d'un adjoint à gauche de l'inclusion $\DeuxCat \hookrightarrow \DeuxCatLax$. La découverte de cette adjonction, dans le cadre plus général des bicatégories, revient à Bénabou, qui semble malheureusement n'avoir rien publié à ce sujet. Le cas particulier de cette construction appliquée aux \deux{}catégories est abordé par Gray dans \cite[I, 4. 23, p. 9]{Gray}. Bien qu'il s'agisse d'une notion « folklorique », sa description concrète semble assez méconnue. Un traitement algébrique détaillé de cette construction se trouve dans \cite[section 1.12]{TheseMoi}. Nous la présentons brièvement — et quelque peu « grossièrement » — ici pour les besoins du présent travail. 

\begin{df}\label{DefTilde}
Étant donné une \deux{}catégorie $\mathdeuxcat{A}$, on note $\TildeLax{\mathdeuxcat{A}}$ la \deux{}catégorie définie comme suit. Ses objets sont les objets de $\mathdeuxcat{A}$. Étant donné deux objets $a$ et $a'$, les \un{}cellules de $a$ vers $a'$ sont les $([m], x : [m] \to \mathdeuxcat{A})$, avec $m \geq 0$ un entier et $x$ un \DeuxFoncteurStrict{} tel que $x_{0} = a$ et $x_{m} = a'$. On pourra noter une telle \un{}cellule par $([m], x_{1,0} : x_{0} \to x_{1}, \dots, x_{m,m-1} : x_{m-1} \to x_{m})$. Étant donné $([m], x)$ et $([n], y)$ deux \un{}cellules parallèles de $a$ vers $a'$, les \deux{}cellules de la première vers la seconde sont de la forme $(\varphi, \alpha_{1}, \dots, \alpha_{n})$ avec $\varphi$ une application croissante de $[n]$ vers $[m]$ telle que $\varphi(0) = 0$, $\varphi(n) = m$ et $x_{\varphi_{i}} = y_{i}$ pour tout objet $i$ de $[n]$, et $\alpha_{i}$ une \deux{}cellule de $x_{\varphi(i), \varphi(i)-1} \dots x_{\varphi(i-1)+1, \varphi(i-1)}$ vers $y_{i, i-1}$ dans $\mathdeuxcat{A}$, pour tout objet $i$ de $[n]$. On pourra noter $(\varphi, \alpha)$ une telle \deux{}cellule ; il s'agit d'une notation plus satisfaisante conceptuellement si l'on considère $\alpha$ comme une \emph{transformation relative aux objets} de $x \varphi$ vers $y$ (Steve Lack a souligné l'importance d'un tel renforcement de la notion de transformation sous le nom d'\emph{icône} ; nous adoptons ce point de vue dans \cite[section 1.12]{TheseMoi} pour décrire l'adjonction de Bénabou). Les diverses unités et compositions se définissent à l'aide de celles de $\mathdeuxcat{A}$. L'identité de l'objet $a$ dans $\TildeLax{\mathdeuxcat{A}}$ est définie par $([0], a)$ (on a noté $a$ l'unique \DeuxFoncteurStrict{} de $[0]$ vers $\mathdeuxcat{A}$ envoyant l'objet $0$ de $[0]$ sur $a$). La composition des \un{}cellules est définie par concaténation, de même que la composition horizontale des \deux{}cellules. La composition verticale des \deux{}cellules est définie comme suit. Étant donné $([m], x)$, $([n], y)$ et $([p], z)$ trois \un{}cellules parallèles de $\TildeLax{\mathdeuxcat{A}}$, $(\varphi, \alpha)$ une \deux{}cellule de $([m], x)$ vers $([n], y)$ et $(\psi, \beta)$ une \deux{}cellule de $([n], y)$ vers $([p], z)$, la composée $(\psi, \beta) \CompDeuxUn (\varphi, \alpha)$ est définie par $(\beta_{p} \CompDeuxUn (\alpha_{n} \CompDeuxZero \dots \CompDeuxZero \alpha_{\psi(p-1) + 1})) \CompDeuxZero \dots \CompDeuxZero (\beta_{1} \CompDeuxUn (\alpha_{\psi(1)} \CompDeuxZero \dots \CompDeuxZero \alpha_{1}))$. En vertu de la loi d'échange, cette composée est égale à $(\beta_{p} \CompDeuxZero \dots \CompDeuxZero \beta_{1}) \CompDeuxUn (\alpha_{n} \CompDeuxZero \dots \CompDeuxZero \alpha_{1})$.

Pour tout \DeuxFoncteurLax{} $u : \mathdeuxcat{A} \to \mathdeuxcat{B}$, on définit un \DeuxFoncteurStrict{} $\TildeLax{u} : \TildeLax{\mathdeuxcat{A}} \to \TildeLax{\mathdeuxcat{B}}$ de la façon suivante. Pour tout objet $a$ de $\TildeLax{\mathdeuxcat{A}}$, 
$$
\TildeLax{u} (a) = u(a)
$$
Pour toute \un{}cellule $([m], x)$ de $\TildeLax{\mathdeuxcat{A}}$, 
$$
\TildeLax{u} ([m], x) = ([m], u(x_{1,0}), \dots, u(x_{m, m-1}))
$$
Pour toute \deux{}cellule $(\varphi, \alpha)$ de $([m], x)$ vers $([n], y)$ comme ci-dessus,
$$
\TildeLax{u} (\varphi, \alpha) = (u(\alpha_{n}) \CompDeuxZero \dots \CompDeuxZero u(\alpha_{1})) (u_{x_{m, m-1}, \dots, x_{\varphi(m-1)+1, \varphi(m-1)}} \CompDeuxZero \dots \CompDeuxZero u_{x_{\varphi(1), \varphi(1)-1}, \dots, x_{1, 0}})
$$
\end{df}

\begin{rem}
En vertu de la loi d'échange, en conservant les mêmes notations, la \deux{}cellule $\TildeLax{u} (\varphi, \alpha)$ est égale à
$$
(u(\alpha_{n}) u_{x_{m, m-1}, \dots, x_{\varphi(m-1)+1, \varphi(m-1)}}) \CompDeuxZero \dots \CompDeuxZero (u(\alpha_{1}) u_{x_{\varphi(1), \varphi(1)-1}, \dots, x_{1, 0}})
$$
\end{rem}

En utilisant la « condition de cocycle généralisée », on vérifie la fonctorialité de la construction ci-dessus. Autrement dit :

\begin{prop}\label{TildeFoncteur}
L'assignation $\mathdeuxcat{A} \mapsto \TildeLax{\mathdeuxcat{A}}$, $u \mapsto \TildeLax{u}$ définit un foncteur de $\DeuxCatLax$ vers $\DeuxCat$.
\end{prop}

Du fait du caractère quelque peu disgracieux de la notation « $\TildeLax{?}$ », on notera $B : \DeuxCatLax{} \to \DeuxCat$ le foncteur défini par la proposition \ref{TildeFoncteur}. On se propose maintenant de vérifier qu'il s'agit d'un adjoint à gauche de l'inclusion canonique $\DeuxCat \hookrightarrow \DeuxCatLax$. 

\begin{df}\label{DefEta}
Pour toute \deux{}catégorie $\mathdeuxcat{A}$, on note $\LaxCanonique{\mathdeuxcat{A}} : \mathdeuxcat{A} \to \TildeLax{\mathdeuxcat{A}}$ le \DeuxFoncteurLax{} défini comme suit. Pour tout objet $a$ de $\mathdeuxcat{A}$, on pose 
$$
\LaxCanonique{\mathdeuxcat{A}} (a) = a
$$
Pour toute \un{}cellule $f$ de $\mathdeuxcat{A}$, on pose 
$$
\LaxCanonique{\mathdeuxcat{A}} (f) = ([1], f)
$$ 
Pour toute \deux{}cellule $\alpha$ de $\mathdeuxcat{A}$, on pose 
$$
\LaxCanonique{\mathdeuxcat{A}} (\alpha) = (1_{[1]}, \alpha)
$$ 
Pour tout objet $a$ de $\mathdeuxcat{A}$, on pose
$$
({\LaxCanonique{\mathdeuxcat{A}}})_{a} = ([1] \to [0], 1_{1_{a}})
$$
Pour tout couple $(f,f')$ de \un{}cellules de $\mathdeuxcat{A}$ telles que la composée $f'f$ fasse sens, on pose
$$
({\LaxCanonique{\mathdeuxcat{A}}})_{f',f} = ([1] \to [2], 1_{f'f})
$$
(La flèche $[1] \to [2]$ est définie de façon unique par la condition qu'il s'agit d'une application croissante respectant les extrémités.)
\end{df}

On laisse au lecteur le soin de vérifier que cela définit bien un \DeuxFoncteurLax{}. 

\begin{df}\label{DefEpsilon}
Pour toute \deux{}catégorie $\mathdeuxcat{A}$, on note $\StrictCanonique{\mathdeuxcat{A}} : \TildeLax{\mathdeuxcat{A}} \to \mathdeuxcat{A}$ le \DeuxFoncteurStrict{} défini comme suit. Pour tout objet $a$ de $\TildeLax{\mathdeuxcat{A}}$, 
$$
\StrictCanonique{\mathdeuxcat{A}} (a) = a
$$ 
Pour toute \un{}cellule $([m], x)$ de $\TildeLax{\mathdeuxcat{A}}$, 
$$
\StrictCanonique{\mathdeuxcat{A}} ([m], x) = x_{m,m-1} \dots x_{1,0}
$$
Pour toute \deux{}cellule $(\varphi, \alpha)$ de $\TildeLax{\mathdeuxcat{A}}$ comme ci-dessus, on pose
$$
\StrictCanonique{\mathdeuxcat{A}} (\varphi, \alpha) = \alpha_{n} \CompDeuxZero \dots \CompDeuxZero \alpha_{1}
$$
\end{df}

On laisse au lecteur le soin de vérifier que cela définit bien un \DeuxFoncteurStrict{}. Les lemmes \ref{UniteNaturelle} et \ref{CouniteNaturelle} se vérifient sans difficulté. 

\begin{lemme}\label{UniteNaturelle}
Pour tout \DeuxFoncteurLax{} $u : \mathdeuxcat{A} \to \mathdeuxcat{B}$, le diagramme dans $\DeuxCatLax{}$
$$
\xymatrix{
\TildeLax{\mathdeuxcat{A}}
\ar[r]^{\TildeLax{u}}
&\TildeLax{\mathdeuxcat{B}}
\\
\mathdeuxcat{A}
\ar[u]^{\LaxCanonique{\mathdeuxcat{A}}}
\ar[r]_{u}
&\mathdeuxcat{B}
\ar[u]_{\LaxCanonique{\mathdeuxcat{B}}}
}
$$
est commutatif.
\end{lemme}

\begin{lemme}\label{CouniteNaturelle}
Pour tout \DeuxFoncteurStrict{} $u : \mathdeuxcat{A} \to \mathdeuxcat{B}$, le diagramme dans $\DeuxCat{}$
$$
\xymatrix{
\TildeLax{\mathdeuxcat{A}}
\ar[r]^{\TildeLax{u}}
\ar[d]_{\StrictCanonique{\mathdeuxcat{A}}}
&\TildeLax{\mathdeuxcat{B}}
\ar[d]^{\StrictCanonique{\mathdeuxcat{B}}}
\\
\mathdeuxcat{A}
\ar[r]_{u}
&\mathdeuxcat{B}
}
$$
est commutatif.
\end{lemme}

Si l'on note $I : \DeuxCat \hookrightarrow \DeuxCatLax$ l'inclusion canonique, les considérations précédentes nous permettent donc de définir des transformations naturelles $\eta : 1_{\DeuxCatLax} \Rightarrow IB$ et $\epsilon : BI \Rightarrow 1_{\DeuxCat}$ dont les composantes en $\mathdeuxcat{A}$ sont $\LaxCanonique{\mathdeuxcat{A}}$ et $\StrictCanonique{\mathdeuxcat{A}}$ respectivement. 

\begin{prop}\label{BIAdjonction}
Le foncteur $B : \DeuxCatLax \to \DeuxCat$ est un adjoint à gauche de l'inclusion $I : \DeuxCat \to \DeuxCatLax$, les transformations naturelles $\TransLaxCanonique$ et $\TransStrictCanonique$ constituant respectivement l'unité et la coünité de l'adjonction $(B,I)$. 
\end{prop}

\begin{proof}
Les identités triangulaires se vérifient sans difficulté.
\end{proof}

\begin{lemme}\label{LemmeTranchesCouniteAspheriques}
Pour toute \deux{}catégorie $\mathdeuxcat{A}$, pour tout objet $a$ de $\mathdeuxcat{A}$, la \deux{}catégorie $\TrancheCoLax{\TildeLax{\mathdeuxcat{A}}}{\StrictCanonique{\mathdeuxcat{A}}}{a}$ admet un objet admettant un objet final.
\end{lemme}

\begin{proof}
Cela résulte de calculs ne présentant aucune difficulté. On présentera plus loin (lemme \ref{StrictInduitAspherique}) les détails de calcul du « cas relatif » plus général. 
\end{proof}

\begin{prop}\label{CouniteEquiFaible}
Pour toute \deux{}catégorie $\mathdeuxcat{A}$, le \DeuxFoncteurStrict{} $\StrictCanonique{\mathdeuxcat{A}} : \TildeLax{\mathdeuxcat{A}} \to \mathdeuxcat{A}$ est une équivalence faible.
\end{prop}

\begin{proof}
Pour tout objet $a$ de $\mathdeuxcat{A}$, la \deux{}catégorie $\TrancheCoLax{\TildeLax{\mathdeuxcat{A}}}{\StrictCanonique{\mathdeuxcat{A}}}{a}$ admet un objet admettant un objet final (lemme \ref{LemmeTranchesCouniteAspheriques}), donc est asphérique (lemme \ref{OFAspherique}). Le résultat s'en déduit en vertu du corollaire \ref{CorollaireThABC}. 
\end{proof}

\begin{prop}\label{UniteEquiFaible}
Pour toute \deux{}catégorie $\mathdeuxcat{A}$, le \DeuxFoncteurLax{} $\LaxCanonique{\mathdeuxcat{A}} : \mathdeuxcat{A} \to \TildeLax{\mathdeuxcat{A}}$ est une équivalence faible.
\end{prop}

\begin{proof}
C'est une section de $\StrictCanonique{\mathdeuxcat{A}}$, qui est une équivalence faible en vertu de la proposition \ref{CouniteEquiFaible}.
\end{proof}

\begin{df}\label{DefBarreLax}
Pour tout \DeuxFoncteurLax{} $u : \mathdeuxcat{A} \to \mathdeuxcat{B}$, on pose $\BarreLax{u} = \StrictCanonique{B} \TildeLax{u}$.
\end{df}

\begin{rem}\label{PropUnivBarre}
En vertu de la théorie classique des adjonctions, étant donné un \DeuxFoncteurLax{} $u : \mathdeuxcat{A} \to \mathdeuxcat{B}$, $\BarreLax{u}$ est l'unique \DeuxFoncteurStrict{} de $\TildeLax{\mathdeuxcat{A}}$ vers $\mathdeuxcat{B}$ rendant le diagramme
$$
\xymatrix{
\TildeLax{\mathdeuxcat{A}}
\ar[dr]
\\
\mathdeuxcat{A} 
\ar[u]^{\LaxCanonique{\mathdeuxcat{A}}}
\ar[r]_{u}
&\mathdeuxcat{B}
}
$$
commutatif.
\end{rem}

\begin{prop}\label{EquiLaxTildeBarreLax}
Un \DeuxFoncteurLax{} $u : \mathdeuxcat{A} \to \mathdeuxcat{B}$ est une équivalence faible si et seulement si $\TildeLax{u}$ en est une, ce qui est le cas si et seulement si $\BarreLax{u}$ en est une.
\end{prop}

\begin{proof}
Il suffit de considérer le diagramme commutatif
$$
\xymatrix{
\TildeLax{\mathdeuxcat{A}}
\ar[r]^{\TildeLax{u}}
&\TildeLax{\mathdeuxcat{B}}
\\
\mathdeuxcat{A}
\ar[u]^{\LaxCanonique{\mathdeuxcat{A}}}
\ar[r]_{u}
&\mathdeuxcat{B}
\ar[u]_{\LaxCanonique{\mathdeuxcat{B}}}
}
$$
et d'invoquer la proposition \ref{CouniteEquiFaible} pour conclure à la première équivalence. La seconde résulte de l'égalité $\BarreLax{u} = \StrictCanonique{\mathdeuxcat{B}} \TildeLax{u}$ (définition \ref{DefBarreLax}) et de la proposition \ref{CouniteEquiFaible}.
\end{proof}

Le théorème \ref{EqCatLocDeuxCatDeuxCatLax} ne sera pas utilisé avant la section \ref{SectionEqCatLoc}.

\begin{theo}\label{EqCatLocDeuxCatDeuxCatLax}
L'inclusion $\DeuxCat \hookrightarrow \DeuxCatLax$ induit une équivalence de catégories entre les catégories localisées $\Localisation{\DeuxCat}{\DeuxLocFond{W}}$ et $\Localisation{\DeuxCatLax}{\DeuxLocFondLaxInduit{W}}$.
\end{theo}

\begin{proof}
C'est une conséquence du fait que les composantes des transformations naturelles $\TransLaxCanonique$ et $\TransStrictCanonique$ sont dans $\DeuxLocFondLaxInduit{W}$ et $\DeuxLocFond{W}$ respectivement.
\end{proof}

\begin{lemme}\label{LemmeDimitri}
Pour tout \DeuxFoncteurLax{} $u : \mathdeuxcat{A} \to \mathdeuxcat{B}$, le diagramme
$$
\xymatrix{
\TildeLax{\mathdeuxcat{A}}
\ar[d]_{\StrictCanonique{\mathdeuxcat{A}}}
\ar[dr]^{\BarreLax{u}}
\\
\mathdeuxcat{A}
\ar[r]_{u}
&\mathdeuxcat{B}
}
$$
est commutatif à une \DeuxTransformationLax{} $\BarreLax{u} \Rightarrow u \StrictCanonique{\mathdeuxcat{A}}$ près. 
\end{lemme}

\begin{proof}
On pose 
$$
\sigma_{a} = 1_{u(a)}
$$ 
pour tout objet $a$ de $\TildeLax{\mathdeuxcat{A}}$ et
$$
\sigma_{([m], x_{1,0}, \dots, x_{m,m-1})} = u_{x_{m,m-1}, \dots, x_{1,0}}
$$
pour tout morphisme $([m], x_{1,0}, \dots, x_{m,m-1})$ dans $\TildeLax{\mathdeuxcat{A}}$. Il résulte des axiomes des \DeuxFoncteursLax{} que cela définit bien une \DeuxTransformationLax{} $\sigma : \BarreLax{u} \Rightarrow u \StrictCanonique{\mathdeuxcat{A}}$. 
\end{proof}

\begin{lemme}\label{StrictInduitAspherique}
Soient $u : \mathdeuxcat{A} \to \mathdeuxcat{B}$ un \DeuxFoncteurLax{}, $b$ un objet de $\mathdeuxcat{B}$ et
$\DeuxFoncTrancheLaxCoq{\StrictCanonique{\mathdeuxcat{A}}}{\sigma}{b} : \TrancheLax{\TildeLax{\mathdeuxcat{A}}}{\BarreLax{u}}{b} \to  \TrancheLax{\mathdeuxcat{A}}{u}{b}$ le \DeuxFoncteurStrict{} induit par le diagramme
$$
\xymatrix{
\TildeLax{\mathdeuxcat{A}}
\ar[dr]^{\BarreLax{u}}
\ar[d]_{\StrictCanonique{\mathdeuxcat{A}}}
\drtwocell<\omit>{<2>\sigma}
\\
\mathdeuxcat{A}
\ar[r]_{u}
&\mathdeuxcat{B}
&.
}
$$
(voir la démonstration du lemme \ref{LemmeDimitri} pour la définition de $\sigma$). Alors, pour tout objet $(a,p)$ de $\TrancheLax{\mathdeuxcat{A}}{u}{b}$, la \deux{}catégorie
$$
\TrancheCoLax{(\TrancheLax{\TildeLax{\mathdeuxcat{A}}}{\BarreLax{u}}{b})}{\DeuxFoncTrancheLaxCoq{\StrictCanonique{\mathdeuxcat{A}}}{\sigma}{b}}{(a,p)}
$$
admet un objet admettant un objet final. 
\end{lemme}

\begin{proof}
Décrivons partiellement la \deux{}catégorie $
\TrancheCoLax{(\TrancheLax{\TildeLax{\mathdeuxcat{A}}}{\BarreLax{u}}{b})}{\DeuxFoncTrancheLaxCoq{\StrictCanonique{\mathdeuxcat{A}}}{\sigma}{b}}{(a,p)}
$.
\begin{itemize}
\item Les objets en sont les quadruplets $(a', p', f, \alpha)$, où $a'$ est un objet de $\mathdeuxcat{A}$, $p' : u(a') \to b$ une \un{}cellule de $\mathdeuxcat{B}$, $f : a' \to a$ une \un{}cellule de $\mathdeuxcat{A}$ et $\alpha : p' \Rightarrow p u(f)$ une \deux{}cellule de $\mathdeuxcat{B}$.
\item Les \un{}cellules de $(a',p', f, \alpha)$ vers $(a'', p'', f', \alpha')$ sont données par les $([m], x, \beta, \gamma)$, avec $([m],x)$ une \un{}cellule de $a'$ vers $a''$ dans $\TildeLax{\mathdeuxcat{A}}$, $\beta : p' \Rightarrow p'' u(x_{m,m-1}) \dots u(x_{1,0})$ une \deux{}cellule de $\mathdeuxcat{B}$ et $\gamma : f' x_{m,m-1} \dots x_{1,0} \Rightarrow f$ une \deux{}cellule dans $\mathdeuxcat{A}$, quadruplets satisfaisant l'égalité
$$
(p \CompDeuxZero (u (\gamma) \CompDeuxUn \DeuxCellStructComp{u}{f'}{x})) \CompDeuxUn (\alpha' \CompDeuxZero u(x_{m,m-1} \dots x_{1,0})) \CompDeuxUn (p'' \CompDeuxZero u_{x}) \CompDeuxUn \beta = \alpha
$$
(On a noté $\DeuxCellStructComp{u}{f'}{x}$ la \deux{}cellule structurale de composition $u_{f', x_{m,m-1} \dots x_{1,0}}$ de $u$, de source $u(f') u (x_{m,m-1} \dots x_{1,0})$ et de but $u(f' x_{m,m-1} \dots x_{1,0})$.) 
\item Si $([m], x, \beta, \gamma)$ et $([n], y, \sigma, \tau)$ sont deux \un{}cellules de $(a',p', f, \alpha)$ vers $(a'', p'', f', \alpha')$, les \deux{}cellules de $([m], x, \beta, \gamma)$ vers $([n], y, \sigma, \tau)$ sont données par les \deux{}cellules $(\varphi, \rho)$ de $([m],x)$ vers $([n],y)$ dans $\TildeLax{\mathdeuxcat{A}}$ satisfaisant 
$$
(p'' \CompDeuxZero (  (u(\rho_{n,n-1}) \CompDeuxZero \dots \CompDeuxZero u(\rho_{1,0})) \CompDeuxUn (u_{x_{\varphi(n-1)} \to \dots \to x_{m}} \CompDeuxZero \dots \CompDeuxZero u_{x_{0} \to \dots \to x_{\varphi(1)}})  )   ) \CompDeuxUn \beta = \sigma
$$
et
$$
\tau \CompDeuxUn (f' \CompDeuxZero \rho_{n} \CompDeuxZero \dots \CompDeuxZero \rho_{1}) = \gamma
$$
\end{itemize}

Dans cette \deux{}catégorie $\TrancheCoLax{(\TrancheLax{\TildeLax{\mathdeuxcat{A}}}{\BarreLax{u}}{b})}{\DeuxFoncTrancheLaxCoq{\StrictCanonique{\mathdeuxcat{A}}}{\sigma}{b}}{(a,p)}$ se distingue l'objet $(a, p, 1_{a}, p \CompDeuxZero \DeuxCellStructId{u}{a})$. 

Soit $(a'', p'' : u(a'') \to b, f' : a'' \to a, \alpha : p'' \Rightarrow p u(f'))$ un objet quelconque de la \deux{}catégorie $\TrancheCoLax{(\TrancheLax{\TildeLax{\mathdeuxcat{A}}}{\BarreLax{u}}{b})}{\DeuxFoncTrancheLaxCoq{\StrictCanonique{\mathdeuxcat{A}}}{\sigma}{b}}{(a,p)}$. Le quadruplet $([1], f', \alpha, 1_{f'})$ définit alors une \un{}cellule de $(a'', p'', f', \alpha)$ vers $(a, p, 1_{a}, p \CompDeuxZero \DeuxCellStructId{u}{a})$. En effet, la condition de commutativité à vérifier se simplifie ici en l'égalité
$$
(p \CompDeuxZero (\DeuxCellStructComp{u}{1_{a}}{f'}  \CompDeuxUn (\DeuxCellStructId{u}{a} \CompDeuxZero u(f')))) \CompDeuxUn \alpha = \alpha,
$$
qui est vérifiée en vertu de l'égalité
$$
\DeuxCellStructComp{u}{1_{a}}{f'}  \CompDeuxUn (  \DeuxCellStructId{u}{a} \CompDeuxZero u(f')  ) = 1_{u(f')}
$$

Supposons donné une \un{}cellule 
$$
([m], x, \alpha' : p'' \Rightarrow p u(x_{m,m-1}) \dots u(x_{1,0}), \gamma : x_{m,m-1}  \dots x_{1,0} \Rightarrow f')
$$
de $(a'', p'', f', \alpha)$ vers $(a, p, 1_{a}, p \CompDeuxZero \DeuxCellStructId{u}{a})$ dans 
$
\TrancheCoLax{(\TrancheLax{\TildeLax{\mathdeuxcat{A}}}{\BarreLax{u}}{b})}{\DeuxFoncTrancheLaxCoq{\StrictCanonique{\mathdeuxcat{A}}}{\sigma}{b}}{(a,p)}
$. La définition des \un{}cellules de cette catégorie assure l'égalité
$$
(p \CompDeuxZero u(\gamma)) \CompDeuxUn (p \CompDeuxZero \DeuxCellStructComp{u}{1_{a}}{x}) \CompDeuxUn (p \CompDeuxZero \DeuxCellStructId{u}{a} \CompDeuxZero u(x_{m,m-1} \dots x_{1,0})) \CompDeuxUn (p \CompDeuxZero u_{x}) \CompDeuxUn \alpha' = \alpha
$$
qui se récrit
$$
(p \CompDeuxZero (u(\gamma) \CompDeuxUn    
\DeuxCellStructComp{u}{1_{a}}{x} \CompDeuxUn    
(\DeuxCellStructId{u}{a} \CompDeuxZero u(x_{m,m-1} \dots x_{1,0}))   \CompDeuxUn  
u_{x})) 
\CompDeuxUn \alpha' = \alpha
$$
La composée
$$
\DeuxCellStructComp{u}{1_{a}}{x} \CompDeuxUn    
(\DeuxCellStructId{u}{a} \CompDeuxZero u(x_{m,m-1} \dots x_{1,0})) 
$$
est une identité. Ainsi, l'égalité 
$$
(p \CompDeuxZero (u(\gamma) \CompDeuxUn u_{x})) \CompDeuxUn \alpha' = \alpha
$$
fait partie des hypothèses. 

Une \deux{}cellule de $([m], x, \alpha', \gamma)$ vers $([1], f', \alpha, 1_{f'})$ dans 
$
\TrancheCoLax{(\TrancheLax{\TildeLax{\mathdeuxcat{A}}}{\BarreLax{u}}{b})}{\DeuxFoncTrancheLaxCoq{\StrictCanonique{\mathdeuxcat{A}}}{\sigma}{b}}{(a,p)}
$
correspond à un couple $(\varphi, \rho)$, où $\varphi : [1] \to [m]$ est un morphisme d'intervalles et $\rho$ une \deux{}cellule de $x_{m,m-1} \dots x_{1,0}$ vers $f'$ dans $\mathdeuxcat{A}$, telles que
$$
(p \CompDeuxZero (u(\rho) \CompDeuxUn u_{x})) \CompDeuxUn \alpha' = \alpha
$$
et
$$
1_{f'} \CompDeuxUn (1_{a} \CompDeuxZero \rho) = \gamma
$$
La seconde égalité force $\rho = \gamma$ et, par hypothèse, la première égalité est vérifiée si l'on fait $\rho = \gamma$. Cela termine la démonstration.
\end{proof}

\begin{lemme}\label{StrictInduitEquiFaible}
Soient $u : \mathdeuxcat{A} \to \mathdeuxcat{B}$ un \DeuxFoncteurLax{} et $b$ un objet de $\mathdeuxcat{B}$. Alors, le \DeuxFoncteurStrict{}
$\DeuxFoncTrancheLaxCoq{\StrictCanonique{\mathdeuxcat{A}}}{\sigma}{b} : \TrancheLax{\TildeLax{\mathdeuxcat{A}}}{\BarreLax{u}}{b} \to  \TrancheLax{\mathdeuxcat{A}}{u}{b}$ (voir l'énoncé du lemme \ref{StrictInduitAspherique}) est une équivalence faible.
\end{lemme}

\begin{proof}
C'est une conséquence immédiate du lemme \ref{StrictInduitAspherique}.
\end{proof}

\begin{lemme}\label{LaxInduitEquiFaible}
Soient $u : \mathdeuxcat{A} \to \mathdeuxcat{B}$ un \DeuxFoncteurLax{} et $b$ un objet de $\mathdeuxcat{B}$. Alors, le  \DeuxFoncteurLax{}
$\DeuxFoncTrancheLax{\LaxCanonique{\mathdeuxcat{A}}}{b} : \mathdeuxcat{A} \to \TildeLax{\mathdeuxcat{A}}$ induit par le diagramme commutatif
$$
\xymatrix{
\TildeLax{A}
\ar[dr]^{\BarreLax{u}}
\\
\mathdeuxcat{A}
\ar[u]^{\LaxCanonique{\mathdeuxcat{A}}}
\ar[r]_{u}
&\mathdeuxcat{B}
}
$$
est une équivalence faible. 
\end{lemme}

\begin{proof}
C'est une section du \DeuxFoncteurStrict{} $\DeuxFoncTrancheLaxCoq{\StrictCanonique{\mathdeuxcat{A}}}{\sigma}{b}$, qui est une équivalence faible (lemme \ref{StrictInduitEquiFaible}). 
\end{proof}

\begin{rem}
Les lemmes \ref{StrictInduitEquiFaible} et \ref{LaxInduitEquiFaible} généralisent les lemmes \ref{CouniteEquiFaible} et \ref{UniteEquiFaible}, sans que leur démonstration en dépende. On aurait donc pu s'abstenir d'énoncer ces cas particuliers. Les calculs s'avérant toutefois sensiblement plus pénibles dans le cas général, il nous paraissait plus naturel, et surtout plus pédagogique, de procéder ainsi. 
\end{rem}

\section{Le cas général}\label{SectionCasGeneral}

\begin{lemme}\label{BarreLaxUV}
Soient $u : \mathdeuxcat{A} \to \mathdeuxcat{B}$ et $v : \mathdeuxcat{B} \to \mathdeuxcat{C}$ deux \DeuxFoncteursLax{}. Alors, $\BarreLax{v} \TildeLax{u} = \BarreLax{vu}$.
\end{lemme}

\begin{proof}
En vertu de la remarque \ref{PropUnivBarre}, comme $\BarreLax{v} \TildeLax{u}$ est un \DeuxFoncteurStrict{}, cela résulte de la suite d'égalités $\BarreLax{v} \TildeLax{u} \LaxCanonique{\mathdeuxcat{A}} = \BarreLax{v} \LaxCanonique{\mathdeuxcat{B}} u = vu$.
\end{proof}

\begin{paragr}
Soient $u$ et $v$ deux \DeuxFoncteursLax{} de $\mathdeuxcat{A}$ vers $\mathdeuxcat{B}$, et $\sigma$ une \DeuxTransformationLax{} (\emph{resp.} une \DeuxTransformationOpLax{}) de $u$ vers $v$. Ces données permettent de définir une \DeuxTransformationLax{} (\emph{resp.} une \DeuxTransformationOpLax{}) $\BarreLax{\sigma}$ de $\BarreLax{u}$ vers $\BarreLax{v}$, comme suit. 

Pour tout objet $a$ de $\mathdeuxcat{A}$, on pose $\BarreLax{\sigma}_{a} = \sigma_{a}$ et $\BarreLax{\sigma}_{([0], a)} = 1_{\sigma_{a}}$. 

Pour toute \un{}cellule $([m], x)$ de $a$ vers $a'$ dans $\TildeLax{\mathdeuxcat{A}}$ avec $m \geq 1$, l'on définit une \deux{}cellule $\sigma_{([m], x)}$ comme suit. Si $m = 1$, $\BarreLax{\sigma}_{([m], x)} = \BarreLax{\sigma}_{([1], x_{1,0})} = \sigma_{x_{1, 0}}$. Si $m \geq 2$,
$$
\BarreLax{\sigma}_{([m], x)} = ((v(x_{m, m-1}) \dots v(x_{2, 1})) \CompDeuxZero \sigma_{x_{1, 0}}) \CompDeuxUn (\BarreLax{\sigma}_{([m-1], (x_{m, m-1}, \dots, x_{2, 1}))} \CompDeuxZero u(x_{1, 0}))
$$
si $\sigma$ est une \DeuxTransformationLax{}, et 
$$
\BarreLax{\sigma}_{([m], (x))} = (\sigma_{x_{m, m-1}} \CompDeuxZero (u(x_{m-1, m-2}) \dots u(x_{1,0}))) \CompDeuxZero (v(x_{m, m-1}) \CompDeuxZero \BarreLax{\sigma}_{([m-1], (x_{m-1, m-2}, \dots, x_{1, 0}))})
$$
si $\sigma$ est une \DeuxTransformationOpLax{}.
\end{paragr}

\begin{lemme}\label{BarreLaxDeuxTrans}
Étant donné deux \DeuxFoncteursLax{} parallèles $u$ et $v$ et une \DeuxTransformationLax{} (\emph{resp.} \DeuxTransformationOpLax{}) $\sigma : u \Rightarrow v$, le procédé décrit ci-dessus définit une \DeuxTransformationLax{} (\emph{resp.} \DeuxTransformationOpLax{}) $\BarreLax{\sigma} : \BarreLax{u} \Rightarrow \BarreLax{v}$.
\end{lemme}

\begin{proof}
Les vérifications sont laissées au lecteur.
\end{proof}

\begin{lemme}\label{CarreCommutatifThALax}
Soient 
$$
\xymatrix{
\mathdeuxcat{A}  
\ar[rr]^{u} 
\ar[dr]_{w} 
&&\mathdeuxcat{B}
\dtwocell<\omit>{<7.3>\sigma} 
\ar[dl]^{v}
\\ 
& 
\mathdeuxcat{C}
&{}
}
$$
un diagramme dans lequel $u$, $v$ et $w$ sont des \DeuxFoncteursLax{} et $\sigma$ est une \DeuxTransformationOpLax{}, $c$ un objet de $\mathdeuxcat{C}$, $\DeuxFoncTrancheLaxCoq{u}{\sigma}{c} : \TrancheLax{\mathdeuxcat{A}}{w}{c} \to \TrancheLax{\mathdeuxcat{B}}{v}{c}$ le \DeuxFoncteurLax{} induit par ces données, et $\DeuxFoncTrancheLaxCoq{\TildeLax{u}}{\BarreLax{\sigma}}{c}$ le \DeuxFoncteurStrict{} induit par le diagramme
$$
\xymatrix{
\TildeLax{\mathdeuxcat{A}}
\ar[rr]^{\TildeLax{u}}
\ar[dr]_{\BarreLax{w}}
&&\TildeLax{\mathdeuxcat{B}}
\dtwocell<\omit>{<7.3>\BarreLax{\sigma}} 
\ar[dl]^{\BarreLax{v}}
\\
&\mathdeuxcat{C}
&{}
}
$$
(voir les lemmes \ref{BarreLaxUV} et \ref{BarreLaxDeuxTrans}). Alors, le diagramme
$$
\xymatrix{
\TrancheLax{\TildeLax{\mathdeuxcat{A}}}{\BarreLax{w}}{c}
\ar[rr]^{\DeuxFoncTrancheLaxCoq{\TildeLax{u}}{\BarreLax{\sigma}}{c}}
&&\TrancheLax{\TildeLax{\mathdeuxcat{B}}}{\BarreLax{v}}{c}
\\
\TrancheLax{\mathdeuxcat{A}}{w}{c}
\ar[u]^{\DeuxFoncTrancheLax{\LaxCanonique{\mathdeuxcat{A}}}{c}}
\ar[rr]_{\DeuxFoncTrancheLaxCoq{u}{\sigma}{c}}
&&\TrancheLax{\mathdeuxcat{B}}{v}{c}
\ar[u]_{\DeuxFoncTrancheLax{\LaxCanonique{\mathdeuxcat{B}}}{c}}
}
$$
est commutatif. 
\end{lemme}

\begin{proof}
Les calculs ne présentent aucune difficulté.
\end{proof}

\begin{corollaire}\label{TexAvery}
Soient 
$$
\xymatrix{
\mathdeuxcat{A}  
\ar[rr]^{u} 
\ar[dr]_{w} 
&&\mathdeuxcat{B}
\dtwocell<\omit>{<7.3>\sigma} 
\ar[dl]^{v}
\\ 
& 
\mathdeuxcat{C}
&{}
}
$$
un diagramme dans lequel $u$, $v$ et $w$ sont des \DeuxFoncteursLax{} et $\sigma$ est une \DeuxTransformationOpLax{} et $c$ un objet de $\mathdeuxcat{C}$. Alors, le \DeuxFoncteurLax{} $\DeuxFoncTrancheLaxCoq{u}{\sigma}{c} : \TrancheLax{\mathdeuxcat{A}}{w}{c} \to \TrancheLax{\mathdeuxcat{B}}{v}{c}$ est une équivalence faible si et seulement si le \DeuxFoncteurStrict{} $\DeuxFoncTrancheLaxCoq{\TildeLax{u}}{\BarreLax{\sigma}}{c} : \TrancheLax{\TildeLax{\mathdeuxcat{A}}}{\BarreLax{w}}{c} \to \TrancheLax{\TildeLax{\mathdeuxcat{B}}}{\BarreLax{v}}{c}$ en est une. 
\end{corollaire}

\begin{proof}
En vertu du lemme \ref{LaxInduitEquiFaible}, les \DeuxFoncteursLax{} $\DeuxFoncTrancheLax{\LaxCanonique{\mathdeuxcat{A}}}{c}$ et $\DeuxFoncTrancheLax{\LaxCanonique{\mathdeuxcat{B}}}{c}$ sont des équivalences faibles. Le lemme \ref{CarreCommutatifThALax} et la saturation faible de la classe des équivalences faibles permettent de conclure.
\end{proof}

\begin{theo}\label{ThALaxTrancheLaxCoq}
Soit 
$$
\xymatrix{
\mathdeuxcat{A}  
\ar[rr]^{u} 
\ar[dr]_{w} 
&&\mathdeuxcat{B}
\dtwocell<\omit>{<7.3>\sigma} 
\ar[dl]^{v}
\\ 
& 
\mathdeuxcat{C}
&{}
}
$$
un diagramme dans lequel $u$, $v$ et $w$ sont des \DeuxFoncteursLax{} et $\sigma$ est une \DeuxTransformationOpLax{}. Supposons que, pour tout objet $c$ de $\mathdeuxcat{C}$, le \DeuxFoncteurLax{} 
$$
\DeuxFoncTrancheLaxCoq{u}{\sigma}{c} : \TrancheLax{\mathdeuxcat{A}}{w}{c} \to \TrancheLax{\mathdeuxcat{B}}{v}{c}
$$ 
soit une équivalence faible. Alors $u$ est une équivalence faible.
\end{theo}

\begin{proof}
En vertu des hypothèses et du corollaire \ref{TexAvery}, pour tout objet $c$ de $\mathdeuxcat{C}$, le \DeuxFoncteurStrict{} $\DeuxFoncTrancheLaxCoq{\TildeLax{u}}{\BarreLax{\sigma}}{c}$ est une équivalence faible. En vertu du théorème \ref{ThAStrictCoq}, $\TildeLax{u}$ est une équivalence faible, donc $u$ est une équivalence faible (proposition \ref{EquiLaxTildeBarreLax}).
\end{proof}

\begin{rem}
L'énoncé du théorème \ref{ThALaxTrancheLaxCoq} admet bien entendu trois versions duales. Elles correspondent aux cas suivants :
\begin{itemize}
\item $u$, $v$ et $w$ sont des \DeuxFoncteursLax{} et $\sigma$ une \DeuxTransformationLax{} de $w$ vers $vu$ ;
\item $u$, $v$ et $w$ sont des \DeuxFoncteursCoLax{} et $\sigma$ une \DeuxTransformationLax{} de $vu$ vers $w$ ;
\item $u$, $v$ et $w$ sont des \DeuxFoncteursCoLax{} et $\sigma$ une \DeuxTransformationOpLax{} de $w$ vers $vu$. 
\end{itemize}
\end{rem}

\section{Les 2-catégories modélisent les types d'homotopie}\label{SectionEqCatLoc}

Comme annoncé, nous démontrons dans cette section que la catégorie localisée de $\DeuxCat$ par les équivalences faibles considérées est équivalente à la catégorie homotopique classique. Nous avons déjà mentionné que Ara et Maltsiniotis \cite{AraMaltsiniotis} utilisent ce résultat pour établir l'existence d'une \emph{équivalence} de Quillen entre $\DeuxCat$ munie de la structure de catégorie de modèles « à la Thomason » qu'ils construisent et la catégorie des ensembles simpliciaux munie de sa structure de catégorie de modèles classique. 

\begin{df}
Étant donné des \DeuxFoncteurLax{} $u$ et $v$ de $\mathdeuxcat{A}$ vers $\mathdeuxcat{B}$, une \DeuxTransformationLax{} (\emph{resp.} \DeuxTransformationOpLax{}) $\sigma : u \Rightarrow v$ sera dite \emph{relative aux objets} si $\sigma_{a} : u(a) \to v(a)$ est une identité pour tout objet $a$ de $\mathdeuxcat{A}$.
\end{df}

\begin{df}
On note 
$$
\underline{Fon}([m], \mathdeuxcat{A})
$$
la catégorie dont les objets sont les \DeuxFoncteursStricts{} de $[m]$ vers $\mathdeuxcat{A}$ et dont les morphismes sont les \DeuxTransformationsLax{} relatives aux objets. Autrement dit, $\underline{Fon}([m], \mathdeuxcat{A})$ est la catégorie 
$$
\coprod_{\substack{(a_{0}, \dots, a_{m}) \in (\Objets{\mathdeuxcat{A}})^{m+1}}} \CatHom{\mathdeuxcat{A}}{a_{0}}{a_{1}} \times \dots \times \CatHom{\mathdeuxcat{A}}{a_{m-1}}{a_{m}}
$$
\end{df}

\begin{df}\label{DefNerfHom}
On note 
$\NerfHom{\mathdeuxcat{A}}$ l'objet simplicial de $\Cat$ défini par 
$$
\begin{aligned}
\NerfHom{\mathdeuxcat{A}} : \Delta^{op} &\to \Cat
\\
m &\mapsto \underline{Fon}([m], \mathdeuxcat{A})
\end{aligned}
$$
et dont les faces et dégénérescences sont définies de façon « évidente ». 

Cela permet de définir un foncteur
$$
\NerfHom : \DeuxCat \to \CatHom{CAT}{\Delta^{op}}{\Cat} 
$$
\end{df}

\begin{df}\label{DefOpIntegrale}
Pour tout foncteur $F : \DeuxCatUnOp{A} \to \Cat$, on pose
$$
{\DeuxInt{}}^{op}F = \left(\DeuxInt{} (?^{op} F)\right)^{op}
$$
où $?^{op}$ désigne l'automorphisme de $\Cat$ qui, à toute catégorie, associe sa catégorie opposée.
\end{df}

Ainsi, ${\DeuxInt{}}^{op}F$ n'est autre que la catégorie notée $\nabla{} F$ dans \cite[section 2.2.6]{THG}. De façon plus explicite, ses objets sont les couples $(a,x)$, avec $a$ un objet de $A$ et $x$ un objet de $F(a)$, et les morphismes de $(a,x)$ vers $(a', x')$ sont les couples $(f, r)$ avec $f : a \to a'$ dans $A$ et $r : x \to F(f)(x')$ dans $F(a)$.

\begin{df}
Soit $\mathdeuxcat{A}$ une \deux{}catégorie. On définit un \DeuxFoncteurLax{}
$$
\SupHomObjet{\mathdeuxcat{A}} : {\DeuxInt{}}^{op} \NerfHom{\mathdeuxcat{A}} \to \mathdeuxcat{A}
$$
par les données suivantes.

Si $([m], (x_{i,i-1})_{1 \leq i \leq m})$ est un objet de ${\DeuxInt{}}^{op} \NerfHom{\mathdeuxcat{A}}$, on pose 
$$
\SupHomObjet{\mathdeuxcat{A}}([m], (x_{i,i-1})_{1 \leq i \leq m}) = x_{m}
$$

Un morphisme de $([m], (x_{i,i-1})_{1 \leq i \leq m})$ vers $([n], (y_{j,j-1})_{1 \leq j \leq n})$ dans ${\DeuxInt{}}^{op} \NerfHom{\mathdeuxcat{A}}$ est de la forme $(\varphi : [m] \to [n], (\alpha_{i})_{1 \leq i \leq m})$, où $\alpha_{i} : x_{i,i-1} \Rightarrow y_{\varphi(i), \varphi(i)-1} \dots y_{\varphi(i-1)+1, \varphi(i-1)}$ est une \deux{}cellule de $\mathdeuxcat{A}$. L'image d'un tel morphisme $(\varphi : [m] \to [n], (\alpha_{i})_{1 \leq i \leq m})$ par $\SupHomObjet{\mathdeuxcat{A}}$ est définie par
$$
\SupHomObjet{\mathdeuxcat{A}} (\varphi : [m] \to [n], (\alpha_{i})_{1 \leq i \leq m}) = y_{n,n-1} \dots y_{\varphi(m)+1, \varphi(m)}
$$

Pour tout objet $([m], (x_{i,i-1})_{1 \leq i \leq m})$ de ${\DeuxInt{}}^{op} \NerfHom{\mathdeuxcat{A}}$, on pose
$$
\DeuxCellStructId{\SupHomObjet{\mathdeuxcat{A}}}{([m], (x_{i,i-1})_{1 \leq i \leq m})} = 1_{1_{x_{m}}} 
$$

Pour tout couple de morphismes composables $(\varphi : [m] \to [n], (\alpha_{i})_{1 \leq i \leq m}) : ([m], (x_{i,i-1})_{1 \leq i \leq m}) \to ([n], (y_{j,j-1})_{1 \leq j \leq n})$ et $(\psi : [n] \to [p], (\beta_{j})_{1 \leq j \leq n}) : ([n], (y_{j,j-1})_{1 \leq j \leq n}) \to ([p], (z_{k,k-1})_{1 \leq k \leq p})$ de ${\DeuxInt{}}^{op} \NerfHom{\mathdeuxcat{A}}$, on pose
$$
\DeuxCellStructComp{\SupHomObjet{\mathdeuxcat{A}}}{(\psi, (\beta_{j})_{1 \leq j \leq n})}{(\varphi, (\alpha_{i})_{1 \leq i \leq m})} = z_{p,p-1} \dots z_{\psi(n)+1, \psi(n)} \CompDeuxZero \beta_{n} \CompDeuxZero \dots \CompDeuxZero \beta_{\varphi(m)+1}
$$
\end{df}

Les propositions \ref{PropDelHoyo} et \ref{SupHomEquiFaible} ne sont autres que \cite[théorème 9.2.4]{TheseDelHoyo} (c'est aussi \cite[théorème 7.3]{ArticleDelHoyo}), aux contextes et notations près. On y renvoie le lecteur pour la démonstration de la proposition \ref{PropDelHoyo}. 

\begin{prop}[Del Hoyo]\label{PropDelHoyo}
Pour toute \deux{}catégorie $\mathdeuxcat{A}$, pour tout objet $a$ de $\mathdeuxcat{A}$, la catégorie
$$
\TrancheLax{\left({\DeuxInt{}}^{op} \NerfHom \mathdeuxcat{A}\right)}{\SupHom_{\mathdeuxcat{A}}}{a}
$$
est asphérique. 
\end{prop}

En vertu du théorème \ref{ThALaxTrancheLaxCoq}, on en déduit la proposition \ref{SupHomEquiFaible}.  

\begin{prop}[Del Hoyo]\label{SupHomEquiFaible}
Pour toute \deux{}catégorie $\mathdeuxcat{A}$, le \DeuxFoncteurLax{} 
$$
\SupHom_{\mathdeuxcat{A}} : {\DeuxInt{}}^{op} \NerfHom \mathdeuxcat{A} \to \mathdeuxcat{A}
$$
est une équivalence faible. 
\end{prop}

\begin{paragr}
On note $\UnLocFond{W}$ la classe des foncteurs (morphismes de $Cat$) dont le nerf est une équivalence faible simpliciale. Autrement dit, $W = \DeuxLocFond{W} \cap \UnCell{\Cat}$. Pour tout morphisme $u : A \to B$ de $\Cat$, on note $\Delta / Nu : \Delta / NA \to \Delta / NB$ l'image de son nerf par le foncteur « catégorie des éléments ». Ce n'est donc rien d'autre que $\DeuxInt{}^{op}{\NerfLax{u}}$ avec les notations introduites plus haut. Si l'on note maintenant (toujours pour distinguer du cas général des \deux{}catégories) $sup_{A} : \Delta / NA \to A$ le foncteur $\SupHom_{A}$, les flèches horizontales du diagramme commutatif
$$
\xymatrix{
\Delta/NA
\ar[d]_{\Delta / Nu}
\ar[rr]^{sup_{A}}
&&A
\ar[d]^{u}
\\
\Delta/NB
\ar[rr]_{sup_{B}}
&&B
}
$$ 
sont des équivalences faibles (il s'agit d'un résultat classique que généralise la proposition \ref{SupHomEquiFaible}).
\end{paragr}

\begin{theo}\label{EqCatLocCatDeuxCat}
L'inclusion $\Cat \hookrightarrow \DeuxCat$ induit une équivalence de catégories entre les catégories localisées $\Localisation{\Cat}{\UnLocFond{W}}$ et $\Localisation{\DeuxCat}{\DeuxLocFond{W}}$. On a donc des équivalences de catégories entre catégories localisées
$$
\Localisation{\DeuxCatLax}{\DeuxLocFondLaxInduit{W}} \simeq \Localisation{\DeuxCat}{\DeuxLocFond{W}} \simeq \Localisation{\Cat}{\UnLocFond{W}} \simeq \Hot
$$
\end{theo}

\begin{proof}
On a déjà montré la première équivalence (voir le théorème \ref{EqCatLocDeuxCatDeuxCatLax}). La troisième est bien connue. Il reste à démontrer la première assertion de l'énoncé. 

Soit $u : \mathdeuxcat{A} \to \mathdeuxcat{B}$ un \DeuxFoncteurStrict{}. Considérons le diagramme
$$
\xymatrix{
{\DeuxInt{}}^{op} \NerfHom \mathdeuxcat{A}
\ar[d]_{{\DeuxInt{}}^{op} \NerfHom {u}}
&&\TildeLax{{\DeuxInt{}}^{op} \NerfHom \mathdeuxcat{A}}
\ar[ll]_{\StrictCanonique{{\DeuxInt{}}^{op} \NerfHom \mathdeuxcat{A}}}
\ar[rr]^{\BarreLax{\SupHom_{\mathdeuxcat{A}}}}
\ar[d]^{\TildeLax{{\DeuxInt{}}^{op} \NerfHom {u}}}
&&\mathdeuxcat{A}
\ar[d]^{u}
\\
{\DeuxInt{}}^{op} \NerfHom \mathdeuxcat{B}
&&\TildeLax{{\DeuxInt{}}^{op} \NerfHom \mathdeuxcat{B}}
\ar[ll]^{\StrictCanonique{{\DeuxInt{}}^{op} \NerfHom \mathdeuxcat{B}}}
\ar[rr]_{\BarreLax{\SupHom_{\mathdeuxcat{B}}}}
&&\mathdeuxcat{B}
}
$$
Les flèches horizontales y figurant sont des équivalences faibles en vertu des propositions \ref{CouniteEquiFaible}, \ref{EquiLaxTildeBarreLax} et \ref{SupHomEquiFaible}. 

Il reste à vérifier la commutativité du diagramme ci-dessus. On remarque d'abord que le carré de gauche est commutatif par naturalité de $\TransStrictCanonique$ (voir le lemme \ref{CouniteNaturelle}). Pour établir la commutativité du carré de droite, il suffit (voir la remarque \ref{PropUnivBarre}) de vérifier l'égalité 
$$
u \phantom{a} \BarreLax{\SupHom_{\mathdeuxcat{A}}} \phantom{a} \LaxCanonique{\DeuxIntOp{\Delta} \NerfHom \mathdeuxcat{A}} = \BarreLax{\SupHom_{\mathdeuxcat{B}}} \phantom{a} \TildeLax{\DeuxIntOp{\Delta} \NerfHom (u)} \phantom{a} \LaxCanonique{\DeuxIntOp{\Delta} \NerfHom \mathdeuxcat{A}}
$$
ce qui se récrit 
$$
u \phantom{a} \SupHom_{\mathdeuxcat{A}} =  \SupHom_{\mathdeuxcat{B}} \phantom{a} \DeuxIntOp{\Delta} \NerfHom (u)
$$
Cette dernière égalité se vérifie directement.

On déduit de cette observation et du rappel précédant l'énoncé du théorème \ref{EqCatLocCatDeuxCat} que l'inclusion $\Cat \hookrightarrow \DeuxCat$ et le foncteur
$$
\begin{aligned}
\DeuxCat &\to \Cat
\\
\mathdeuxcat{A} &\mapsto {\DeuxInt{}}^{op} \NerfHom \mathdeuxcat{A}
\\
u &\mapsto {\DeuxInt{}}^{op} \NerfHom (u)
\end{aligned}
$$
induisent des équivalences quasi-inverses l'une de l'autre entre les catégories localisées de $\Cat$ et $\DeuxCat$ relativement à $W$ et $\DeuxLocFond{W}$ respectivement.
\end{proof}

\begin{rem}\label{RemarqueDelHoyoAuraitPu}
Comme nous l'avons dit, l'équivalence $\Localisation{\DeuxCat}{\DeuxLocFond{W}} \simeq \Localisation{\Cat}{\UnLocFond{W}}$ peut en fait se déduire des résultats figurant dans la thèse de del Hoyo \cite{TheseDelHoyo}. Il suffit en effet pour cela d'utiliser l'analogue de la construction universelle de Bénabou que nous présentons, mais pour les seuls \DeuxFoncteursLax{} \emph{normalisés} de source une \un{}catégorie \cite[définition 7.5.8]{TheseDelHoyo}, et de remplacer la coünité de l'adjonction de Bénabou par son analogue dans ce dernier cadre (voir la démonstration de \cite[proposition 8.6.1]{TheseDelHoyo}). Le lecteur adaptera sans difficulté la démonstration que nous présentons. Il prendra garde au fait que les foncteurs lax de del Hoyo correspondent à ce que nous appelons — suivant en cela l'usage dominant — foncteurs colax normalisés. Del Hoyo démontre donc dans \cite[théorème 9.3.5]{TheseDelHoyo} que la localisation de la catégorie des \deux{}catégories et des foncteurs lax normalisés par les foncteurs lax normalisés qui sont des équivalences faibles est équivalente à la localisation de $\Cat$ par les équivalences faibles, donc à $\Hot$. En somme, on dispose maintenant des modèles catégoriques suivants des types d'homotopie : la catégorie des catégories et foncteurs, la catégorie des \deux{}catégories et \DeuxFoncteursStricts{}, la catégorie des \deux{}catégories et \DeuxFoncteursLax{} (et, bien sûr, celle des \deux{}catégories et \DeuxFoncteursCoLax{}), la catégorie des \deux{}catégories et \DeuxFoncteursLax{} normalisés (et, bien sûr, celle des \deux{}catégories et \DeuxFoncteursCoLax{} normalisés). 
\end{rem}

%
%

\end{document}